\documentclass[11pt]{amsart}
\usepackage[utf8]{inputenc}
\usepackage{tikz-cd} 
\usepackage{amsmath}
\usepackage{amsthm}
\usepackage{amsfonts}
\usepackage[margin=1.2in]{geometry}
\usepackage{mathdots}
\usepackage{hyperref}
\hypersetup{colorlinks=true,allcolors=blue}
\usepackage[f]{esvect}
\usepackage{enumerate}

 \usepackage{amssymb}

 \usepackage{amsthm}
 \usepackage{amsmath}
 \usepackage{natbib}
 \usepackage{tikz}
 \usepackage{graphicx}
 \usepackage{epstopdf}
 \usepackage{xurl}
 \usepackage{multirow}
 \usepackage{xcolor}
 \usepackage[table]{xcolor}
 \usetikzlibrary{patterns}
\counterwithin{figure}{section}
\counterwithin{table}{section}
\usepackage{orcidlink}

\newtheorem{theorem}{Theorem}[section]
\newtheorem{definition}[theorem]{Definition}
\newtheorem{lemma}[theorem]{Lemma}
\newtheorem{proposition}[theorem]{Proposition}
\newtheorem{corollary}[theorem]{Corollary}
\newtheorem{example}[theorem]{Example}

\newcommand{\im}{\textrm{Im}}
\newcommand{\Q}{\mathcal{Q}}
\newcommand{\A}{\mathcal{A}}
\newcommand{\As}{{A^{\sigma}}}
\newcommand{\B}{\mathcal{B}}
\newcommand{\C}{\mathcal{C}}
\newcommand{\R}{\mathcal{R}}
\newcommand{\cE}{\mathcal{E}}

\newcommand{\ival}[1]{\langle #1 \rangle}

\begin{document}

\title{Dense alternating sign matrices and imprecise copulas}

%\date{\today}

\author[T. Ko\v{s}ir]{Toma\v{z} Ko\v{s}ir \orcidlink{0000-0002-3661-8094} }
\address{University of Ljubljana, Faculty of Mathematics and Physics, Jadranska ulica 19, 1000 Ljubljana, Slovenia, and Institute of Mathematics, Physics and Mechanics, Jadranska ulica 19, 1000 Ljubljana, Slovenia} 
\email{tomaz.kosir@fmf.uni-lj.si}

\author[E. Perrone]{Elisa Perrone \orcidlink{0000-0003-0370-9835}}
\address{Eindhoven University of Technology, Department of Mathematics and Computer Science, Groene Loper 5, 5612 AZ, Eindhoven, The Netherlands} 
\email{e.perrone@tue.nl}

\author[N. Stopar]{Nik Stopar \orcidlink{0000-0002-0004-4957} } 
\address{University of Ljubljana, Faculty of Civil and Geodetic Engineering, Jamova cesta 2, 1000 Ljubljana, Slovenia, University of Ljubljana, Faculty of Mathematics and Physics, Jadranska ulica 19, 1000 Ljubljana, Slovenia, and Institute of Mathematics, Physics and Mechanics, Jadranska ulica 19, 1000 Ljubljana, Slovenia}
\email{nik.stopar@fgg.uni-lj.si}

\begin{abstract}
 {In this paper, we study discrete quasi-copulas and discrete imprecise copulas of minimal range, which naturally correspond to alternating sign matrices. We identify dense alternating sign matrices as a particularly well-behaved subclass. We prove that the associated discrete imprecise copulas are always coherent and hence avoid sure loss. We further show that a particular extension of the corresponding discrete quasi-copulas to the full domain preserves this structure and yields self-dual imprecise copulas. To the best of our knowledge, this provides the %first result 
 {first general result %(we have one $7\times 7$ example already in \cite{K-BKOS})
} in the literature showing how properties established for discrete imprecise copulas can be preserved when extending them to their full-domain counterpart. We conclude the paper by showing that, for general alternating sign matrices, the associated discrete imprecise copulas need not avoid sure loss, further emphasizing the special role of the subclass of dense alternating sign matrices. We establish, however, that the class of discrete quasi-copulas of minimal range is invariant under all defect transformations.}

\end{abstract}

\thanks{The first author acknowledges financial support from the Slovenian Research and Innovation Agency (by research program No. P1-0448 until December 31, 2025, and by research program No. P1-0222 and research project No. J1-50002 after January 1, 2026). The third author acknowledges financial support from the Slovenian Research and Innovation Agency (by research program No. P1-0222 and research projects No. J1-70034 and J1-50002).}


\maketitle

\section{Introduction}

The theory of imprecise probability, initiated by Walley \cite{Walley}, provides a framework for modeling uncertainty when precise probabilistic assessments are not available. Its applications span several areas, e.g., in decision theory \cite{JSA,MMM,Troff}, reliability theory \cite{BEER2013,Coolen,GRAY-et-al_2024,OKS,UtCo,YDSS}, financial risk management \cite{ADEH,Delbaen,Vicig}, game theory \cite{MiMo,Nau},  {and recently to neural networks \cite{Pi-EsGoPu26}}. Imprecise probability of an event $A$ is described as an interval $[\underline{P}(A),\overline{P}(A)]$ containing possible values of the exact probability $P(A)$. The lower probability $\underline{P}(A)$ captures the evidence supporting $A$, while the upper probability $\overline{P}(A)$ expresses the lack of evidence against $A$ (here  {it is only assumed that $\underline{P}$ and $\overline{P}$ are} monotone on sets rather than being additive measures). Modeling imprecision of random variables is based on the notion of a \textit{$p$-box} (a probability box) \cite{FKGMS,FeTu,PVMM-2,TroDe,TMD,UtDe}. For a single random variable $X$, a $p$-box $(\underline{F},\overline{F})$, where $\underline{F}(x)$ and $\overline{F}(x)$ is a pair of cumulative distribution functions  such that $\underline{F}(x)\leq\overline{F}(x)$ for all $x\in\mathbb{R}$, is the set of all cumulative distribution functions $F(x)$ such that $\underline{F}(x)\leq F(x)\leq \overline{F}(x)$ for all $x\in\mathbb{R}$, where $\underline{F}(x)$ and $\overline{F}(x)$ represent the lower and upper bound for the probability of the event $[X\leq x]$. The point-wise infimum and supremum of any $p$-box, also called \textit{envelopes}, are exactly its lower and upper bound. Thus, they remain within the class of cumulative distribution functions. The situation changes dramatically for distribution functions of two or more variables, where the envelopes of a $p$-box may not be cumulative distribution functions \cite{PVMM-2}.
Hence, one of the main questions in the literature is whether one can extend the classical Sklar's theorem \cite{Sklar} to the imprecise setting. To deal with the question a notion of imprecise copula was introduced as a pair of quasi-copulas $(\underline{Q}(u,v),\overline{Q}(u,v))$ that satisfies certain properties \cite{PVMM}. See \cite{MMPV,FullSklar,MultiSklar} for recent results on the imprecise Sklar's theorem and \cite{Stop} for an extended discussion of the topic.

\smallskip
Two central notions in the theory of imprecise probability are related to existence of a copula in an imprecise copula. For an imprecise copula $(\underline{Q},\overline{Q})$, we denote by $\C$ the set of all proper bivariate copulas $C$ such that $\underline{Q}(x,y)\leq C(x,y)\leq \overline{Q}(x,y)$ for all $x,y$ in the domain. Then an imprecise copula $(\underline{Q},\overline{Q})$  {\emph{avoids sure loss}} if $\C\neq\emptyset$ \cite[Prop. 6]{PVMM-2} and \emph{is coherent} if the point-wise infimum of $\C$ is equal to $\underline{Q}$ and the point-wise supremum of $\C$ is equal to $\overline{Q}$ \cite[Prop. 9]{PVMM-2}. 
Omladi\v c and Stopar \cite[Example 18]{FinalS} provide examples of a bivariate imprecise copula both in discrete and full domain cases that \emph{do not avoid sure loss} and so they are not coherent either. They also characterize imprecise copulas that {avoid} sure loss (see \cite[Prop. 16]{FinalS} for the discrete case and \cite[Thm. 17]{FinalS} for the general case). In doing so, they give a negative answer to the question  {whether each imprecise copula avoids sure loss} posed in \cite{MMPV}. Some of these results are based on six transformations on quasi-copulas that were introduced in Dibala et al. \cite{DS-PMK}.
Two of them, the \textit{main} and the \textit{opposite defect} transformations, give rise to two imprecise copulas when applied to a quasi-copula (see \cite[Thm. 7]{FinalS}). 

In recent years, discrete imprecise copulas have received increasing attention in the literature. In particular, Ko\v sir and Perrone \cite{KoPe-DIC} investigated several properties of these functions.  {Their connection with discrete geometry and combinatorics provides additional tools for their analysis and insights that can be transferred to the full-domain setting. Indeed, } discrete copulas and quasi-copulas admit natural matrix representations \cite{AST-1,F-SQ-MU-F,KolMor,MST,Q-MS}, their study is closely related to the geometry of the corresponding convex polytopes, i.e., compact convex spaces with a finite number of extreme points. More precisely, for uniform grid domains, discrete copulas correspond to Birkhoff polytopes \cite{Ziegler_95}, while discrete quasi-copulas correspond to alternating sign matrix (ASM) polytopes \cite{striker_2009} (see \cite{Perrone2022,Perrone2021,PSU} and references therein for further details).
These polytopes, respectively, have the permutation matrices and the alternating sign matrices as their extreme points \cite{striker_2009}. 
Recall that an ASM is a matrix with entries in $\{-1,0,1\}$ such that the nonzero entries in each row and column alternate in sign and each row and column sums up to $1$. The set of all $n\times n$ ASMs is the Dedekind-MacNeille completion of the set of all $n\times n$ permutation matrices with respect to the Bruhat partial order \cite{LascouxSchutzenberger1996}.
On the quasi-copula side, ASMs correspond to discrete copulas and quasi-copulas of minimal range, which  {are the central} object of study in this paper. Here, minimal range means that these discrete (quasi-)copulas take values only in $\{0,\frac{1}{n},\frac{2}{n},\ldots,\frac{n-1}{n},1\}$, where $n+1$ denotes the number of grid points in each coordinate direction.

 {One of the objectives of this paper is to further clarify the special role played by discrete quasi-copulas of minimal range and discrete imprecise copulas of minimal range (called \emph{ASM imprecise copulas}). In particular, we show that this class of quasi-copulas is invariant under all six defect transformations introduced in \cite{DS-PMK}. Our main focus is on the subclass of dense ASMs, which is stable under the main and opposite defect transformations. The corresponding ASM imprecise copulas were shown in \cite[Thm. 4.11]{KoPe-DIC} to provide examples of self-dual imprecise copulas. Here, we show that they are always coherent, and hence avoid sure loss.
In the second part of the paper, we consider extensions of discrete quasi-copulas and discrete imprecise copulas to the full domain. We show that the extensions corresponding to dense ASMs yield self-dual imprecise copulas on the full domain, providing the first general result on properties of discrete imprecise copulas that are preserved under extension to their full-domain counterparts. Finally, we give an explicit construction showing that ASM imprecise copulas do not, in general, avoid sure loss, further highlighting the special role of dense ASMs and complementing the examples of \cite{FinalS}.}

\smallskip
 {The paper is structured as follows. In Section \ref{sec:preliminaries}, we introduce the basic notions of (discrete) copulas, quasi-copulas, and the corresponding matrices. In Section \ref{sec:defect_transform}, we study defect transformations and show that they leave the set of quasi-copulas of minimal range invariant. In Section \ref{five}, we focus on dense ASMs and show that the corresponding discrete imprecise copulas are always coherent. In Section \ref{sect:extensions}, we study extensions of discrete quasi-copulas to the full domain and their defects. In Section \ref{sect:ext_dense}, we specialize these results to dense ASMs and show that the corresponding extensions yield self-dual imprecise copulas on the full domain. In Section \ref{DIC}, we show that the favorable behavior of dense ASMs does not extend to general ASMs: for every $ n \geq 17$, we construct a self-dual ASM imprecise copula of minimal range that does not avoid sure loss. The last section contains some concluding remarks.}

\section{Preliminaries}
\label{sec:preliminaries}

In this section, we recall the most relevant results on copulas and quasi-copulas in the discrete setting. Here we restrict ourselves to the special case of uniform grid domains, that is, we consider a uniform partition of the unit interval $I_n=\{ 0, 1/n, \ldots, (n-1)/n, 1 \}$ for $n\in\mathbb{N}$ and define discrete (quasi-)copulas on the square grid $I_n^2$.
In such a case, a \emph{discrete copula} is a function $C: I_n^2 \rightarrow [0,1]$ that satisfies the following conditions:
\begin{enumerate}
\item[(C1')] $C(x,0)=C(0,y)=0$ and $C(x,1)=x$, $C(1,y)=y$ for every $x,y \in I_n$;
	\vspace{3pt}
\item[(C2')] $C$ is \emph{supermodular} or $2$-\emph{increasing} on any rectangle, i.e., 
$$C(x_1,y_1) + C(x_2,y_2) \geq C(x_1,y_2) + C(x_2,y_1)$$ 
for every $x_1, x_2, y_1, y_2 \in I_n$ such that $x_1 \leq x_2$ and $y_1 \leq y_2$.
\end{enumerate}

A \emph{discrete quasi-copula} is a function $Q$ on $I_n^2$ that satisfies the following conditions:
\begin{enumerate}
    \item[(Q1')] $Q(x,0)=Q(0,y)=0$ and $Q(x,1)=x$, $Q(1,y)=y$ for every $x,y\in I_n$;
    \item[(Q2')] $Q$ is \emph{supermodular} on any rectangle with at least one edge on the boundary of the unit square, i.e., 
    \[Q(x_1,y_1) + Q(x_2,y_2) \geq Q(x_1,y_2) + Q(x_2,y_1)\]
    for every $x_1, x_2, y_1, y_2 \in I_n$ such that $x_1\leq x_2$, $y_1\leq y_2$ and at least one of $x_1$, $x_2$, $y_1$, $y_2$ is either equal to 0 or to 1. 
\end{enumerate}

\smallskip
Note that by results of \cite{genest_characterization_1999} quasi-copulas can be charactarized also as the functions satisfying (Q1') and   
\begin{enumerate}
    \item[(Q2'$^{\dagger})$] $Q$ is \emph{increasing} in each variable;
	\vspace{3pt}
    \item[(Q3'$^{\dagger}$)] $Q$ satisfies the $1$-\emph{Lipschitz condition}, i.e.,
     $$ |Q(x_2,y_2) -Q(x_1,y_1)| \leq |x_2 - x_1| + |y_2 - y_1|$$
     for every $x_1,x_2,y_1,y_2 \in I_n.$ 
   
\end{enumerate}

\smallskip
We denote by $\Q'_n$ the set of discrete quasi-copulas on $I_n^2$ and, respectively, by $\C'_n$ the set of all discrete copulas on $I_n^2$. 

We notice that any discrete copula or quasi-copula $Q': I_n^2 \rightarrow [0,1]$ can be expressed as a map $Q:L_n^2 \rightarrow [0,n]$, with $L_n=\{0,1,\ldots,n\}$, through the following transformation:
\begin{equation}
\label{eq:Qprime}
Q(r,s)=n\cdot Q'\left(\frac{r}{n},\frac{s}{n}\right) \textrm{ for } r,s\in L_n.    
\end{equation}
Hence, we denote by $\Q_n$ the set of all functions on $L_n^2$ that correspond to discrete quasi-copulas and by $\C_n$ the set of all functions on $L_n^2$ that correspond to discrete copulas. Functions in $\C_n$ and $\Q_n$ satisfy conditions on $L_n^2$ that are analogous to conditions $(C1')$-$(C2')$ and $(Q1')$-$(Q2')$, respectively. We denote these conditions by $(C1)$-$(C2)$ and $(Q1)$-$(Q2)$, respectively.
 { 
\begin{enumerate}
\item[(C1)] $C(x,0)=C(0,y)=0$ and $C(x,n)=x$, $C(n,y)=y$ for every $x,y \in L_n$;
	\vspace{3pt}
\item[(C2)] $C$ is \emph{supermodular} or $2$-\emph{increasing} on any rectangle, i.e., 
$$C(x_1,y_1) + C(x_2,y_2) \geq C(x_1,y_2) + C(x_2,y_1)$$ 
for every $x_1, x_2, y_1, y_2 \in L_n$ such that $x_1 \leq x_2$ and $y_1 \leq y_2$.
    \item[(Q1)] $Q(x,0)=Q(0,y)=0$ and $Q(x,n)=x$, $Q(n,y)=y$ for every $x,y\in L_n$;
    \item[(Q2)] $Q$ is \emph{supermodular} on any rectangle with at least one edge on the boundary, i.e., 
    \[Q(x_1,y_1) + Q(x_2,y_2) \geq Q(x_1,y_2) + Q(x_2,y_1)\]
    for every $x_1, x_2, y_1, y_2 \in L_n$ such that $x_1\leq x_2$, $y_1\leq y_2$ and at least one of $x_1$, $x_2$, $y_1$, $y_2$ is either equal to 0 or to $n$. 
\end{enumerate}
}

\smallskip
 {Note that by results of \cite{genest_characterization_1999} discrete quasi-copulas on $L_2^n$ can be charactarized also as the functions satisfying (Q1) and   
\begin{enumerate}
    \item[(Q2$^{\dagger})$] $Q$ is \emph{increasing} in each variable;
	\vspace{3pt}
    \item[(Q3$^{\dagger}$)] $Q$ satisfies the $1$-\emph{Lipschitz condition}, i.e.,
     $$ |Q(x_2,y_2) -Q(x_1,y_1)| \leq |x_2 - x_1| + |y_2 - y_1|$$
     for every $x_1,x_2,y_1,y_2 \in L_n.$ 
\end{enumerate}
}

Quasi-copulas and copulas on $I_n^2$ (or on $L_n^2$) can naturally be identified with square matrices of size $n$. 
In particular, we denote the matrix of $Q'\in\Q'_n$ or  {$Q\in\Q_n$, respectively,} by the same letter $Q'$  {or $Q$, respectively,} and write  
$$Q'=\begin{pmatrix}
 q'_{11} & q'_{12} & \cdots  & q'_{1n} \\
 q'_{21} & q'_{22} & \cdots  & q'_{2n} \\
 \vdots & \vdots & & \vdots\\
 q'_{n1} & q'_{n2} & \cdots  & q'_{nn} 
\end{pmatrix},\ \ \mathrm{where}\ q'_{rs}= Q'\left(\frac{r}{n},\frac{s}{n}\right)$$
 {or
$$
Q=\begin{pmatrix}
 q_{11} & q_{12} & \cdots  & q_{1n} \\
 q_{21} & q_{22} & \cdots  & q_{2n} \\
 \vdots & \vdots & & \vdots\\
 q_{n1} & q_{n2} & \cdots  & q_{nn} 
\end{pmatrix},\ \ \mathrm{where}\ q_{rs}= Q\left({r},{s}\right),
$$
respctively.}
Here, we omit the values at $r=0$ and $s=0$ that are all $0$. In addition, we preserve the usual directions of the rows and columns in the matrices, which correspond to the usual coordinate system that is rotated by $90^{\circ}$ in a negative direction.

\begin{example}\label{example2.1}
Consider a quasi-copula  {$Q'$ on $I_n^2$ given by}
\renewcommand{\arraystretch}{1.3}
$$Q'=\begin{pmatrix}
        0 & \frac14 & \frac14 & \frac14\\
        \frac14 & \frac14 & \frac12 & \frac12 \\
        \frac14 & \frac12 & \frac12 & \frac34 \\
        \frac14 & \frac12 & \frac34 &1
        \end{pmatrix}.$$ 
The corresponding map $Q$ on $L_4$ is then given by the matrix
\renewcommand{\arraystretch}{1}
$$Q=\begin{pmatrix}
        0 & 1 & 1 & 1\\
        1 & 1 & 2 & 2\\
        1 & 2 & 2 & 3\\
        1 & 2 & 3 & 4
\end{pmatrix}.$$ 
\end{example}

We explain below that discrete quasi-copulas on $L_n^2$ correspond to the points of the \textit{Alternating Sign Matrices (ASM) polytope} $\A^{B}_n$ \cite{AST-1}. This polytope, studied in \cite{striker_2009}, is the space of all \textit{alternating bistochastic matrices} (ABM), i.e., square matrices that generalize bistochastic matrices by allowing for negative entries. Specifically, Theorem~2.1 of \cite{striker_2009} states that a matrix 
\begin{equation}\label{matrix A}
A=\begin{pmatrix}
 a_{11} & a_{12} & \cdots  & a_{1n} \\
 a_{21} & a_{22} & \cdots  & a_{2n} \\
 \vdots & \vdots & & \vdots\\
 a_{n1} & a_{n2} & \cdots  & a_{nn} 
\end{pmatrix}
\end{equation}
belongs to the Alternating Sign Matrix Polytope, i.e., $A$ is an \emph{ABM}, if, for all $i, j =1, \ldots, n$, it satisfies the following relations:
\begin{equation}\label{ABM_conditions}
0\leq\sum_{s=1}^{i}a_{sj}\leq 1,
\quad
0\leq\sum_{r=1}^{j}a_{ir}\leq 1, \quad \sum_{s=1}^na_{sj}=1, \quad
\sum_{r=1}^na_{ir}=1.
\end{equation}
The vertices of $\A^{B}_n$ are the alternating sign matrices (ASM), which are alternating bistochastic matrices with entries equal to $-1, 0,$ or $1$. Conditions (\ref{ABM_conditions}) imply that in each row and each column of an ASM the nonzero entries alternate in sign. We denote the set of $n \times n$ alternating sign matrices by $\A_n$. By definition, in each row and each column of an ASM the first and the last nonzero entry are equal to $1$. In particular, all permutation matrices are ASMs. We say that $A\in\A_n$ is a \emph{proper ASM} if it contains at least one entry equal to $-1$, and that a matrix is \textit{nonnegative} if all its entries are nonnegative.  
 
\medskip

We note that in the terminology used in the theory of alternating sign matrices and alternating bistochastic matrices, discrete quasi-copulas over $L_n$ are called \emph{corner sum matrices} \cite{BiKo,BS,Propp,RoRu}. 
Indeed, given $A\in\A^B_n$, we can construct the unique quasi-copula $Q'\in \Q'_n$ associated with $A$ (see \cite[Prop. 3 \& Cor. 1]{AST-1}) as follows:
\begin{equation}
\label{eq:bij1}
Q'\left(\frac{r}{n},0\right)=Q'\left(0,\frac{s}{n}\right)=0,\ \textrm{ for } r,s\in L_n,\textrm{ and }
Q'\left(\frac{r}{n},\frac{s}{n}\right)=\frac{1}{n}\sum_{i=1}^r\sum_{j=1}^s a_{ij}, \textrm{ for }r,s\ge 1,
\end{equation}
and the unique $Q\in\Q_n$ by
\begin{equation}
\label{eq:bij2}
Q\left({r},0\right)=Q\left(0,{s}\right)=0,\ \textrm{ for } r,s\in L_n,\textrm{ and }
Q\left({r},{s}\right)=\sum_{i=1}^r\sum_{j=1}^s a_{ij}, \textrm{ for }r,s\ge 1.
\end{equation}
Conversely, for any $Q'\in\Q'_n$  {or $Q\in\Q$, respectively,} the uniquely associated matrix $A=\left[a_{st}\right]_{s,t=1}^n $ in $\A^B_n$ is given by
 {
\begin{equation}
\label{eq:transf}
a_{st}=n\left(Q'\left(\frac{s}{n},\frac{t}{n}\right)+Q'\left(\frac{s-1}{n},\frac{t-1}{n}\right)-Q'\left(\frac{s}{n},\frac{t-1}{n}\right)-Q\left(\frac{s-1}{n},\frac{t}{n}\right)\right).
\end{equation}
or}
\begin{equation}
\label{eq:transf-2}
a_{st}=Q(s,t)+Q(s-1,t-1)-Q(s,t-1)-Q(s-1,t),
\end{equation}
 {respectively.}

To emphasize the correspondence between $A$ and $Q$, we sometimes write $A=A(Q)$ and $Q=Q(A)$. The one-to-one relation can be transferred to $Q'\in\Q'_n$ by dividing all values in $Q$ by $n$. 
Furthermore, the same bijective maps link the elements of $\C_n$ or $\C'_n$  with the elements of \emph{the Birkhoff polytope} $\B_n$ of all bistochastic matrices.

\begin{example}\label{example2.2}
    Consider again Example \ref{example2.1}. The matrix associated to $Q'$ 
    or to $Q$ is a proper ASM
    \renewcommand{\arraystretch}{1}
    $$A(Q)=\begin{pmatrix}
        0 & 1 & 0 &  0\\
        1 & -1 & 1 & 0\\
        0 & 1 & -1 & 1 \\
        0 & 0 & 1 & 0
\end{pmatrix}.$$
    
\end{example}

 {To avoid fractions in matrices as $Q'$ in Example \ref{example2.1}, we will hereafter in the paper work with alternating bistochastic matrices. In the discrete setting the corresponding copulas and quasi-copulas will be defined on $L_n^2$. This will change when considering extensions to full domain in Sections \ref{sect:extensions} and \ref{sect:ext_dense}. }

\medskip

Next, we consider the set $\R$ of all rectangles with vertices on the grid $L_n\times L_n$. So, $\R$ is the set of all rectangles $[i,j]\times [k,l]$, where $0\leq i\leq j\leq n$ and $0\leq k\leq l\leq n$, and $[i,j]=\{i,i+1,\ldots,j\}$ for $i\leq j$. For a rectangle $R=[i,j]\times [k,l]$ we define its $Q$-volume by
$$V_{Q}(R)=Q(i,k)+Q(j,l)-Q(j,k)-Q(i,l).$$
For the corresponding rectangle $R'=[\tfrac{i}{n},\tfrac{j}{n}]\times [\tfrac{k}{n},\tfrac{l}{n}]$
with vertices on the grid $I_n \times I_n$ we define its $Q'$-volume by
$$V_{Q'}(R')=Q'\left(\frac{i}{n},\frac{k}{n}\right)+Q'\left(\frac{j}{n},\frac{l}{n}\right)-Q'\left(\frac{j}{n},\frac{k}{n}\right)-Q'\left(\frac{i}{n},\frac{l}{n}\right).$$

Note that when we take an ASM or another matrix $A\in\A_n^B$ as in \eqref{matrix A}, then 
\begin{equation}\label{a_ijASM}
a_{ij}=V_{Q}([i-1,i]\times [j-1,j])=nV_{Q'}([\tfrac{i-1}{n},\tfrac{i}{n}]\times[\tfrac{j-1}{n},\tfrac{j}{n}])
\end{equation}
for $Q'=Q(A)$ and $Q$ is the quasi-copula related to $Q'$ as in \eqref{eq:Qprime}. In a similar way, we have 
\begin{equation}\label{V_Q(A)ASM}
    V_{Q}(R)=nV_{Q'}(R')=\sum_{r=i+1}^{j}\sum_{s=k+1}^{l}a_{rs}
\end{equation}
for any rectangle $R=[i,j]\times[k,l]\in\R$ and the corresponding rectangle $R'=[\tfrac{i}{n},\tfrac{j}{n}]\times[\tfrac{k}{n},\tfrac{l}{n}]$.

 {
\begin{example}
    Consider two ASMs:
$$
   A= \begin{pmatrix}
                        0 & 0  & 0 & 1 & 0  & 0 \\
                        0  & 0 & 1 & 0 & 0 & 0\\
                        0 & 1 & 0 & -1 & 0 & 1\\
                        1 & 0 & -1 & 0 & 1 & 0\\
                        0 & 0 & 0 & 1 & 0 & 0\\
                        0 & 0 & 1 & 0 & 0 & 0 \\                        
\end{pmatrix}$$
\textrm{ and } 
$$B=\begin{pmatrix}
0 & 0 & 0 & 1 & 0 & 0 \\
0 & 0 & 1 & 0 & 0 & 0 \\
0 & 1 & 0 & -1 & 1 & 0 \\
1 & 0 & -1 & 1 & -1 & 1 \\
0 & 0 & 1 & 0 & 0 & 0 \\
0 & 0 & 0 & 0 & 1 & 0
\end{pmatrix}.
$$
Some of the volumes for given rectangles corresponding to the quasi-copulas $Q(A)$ and $Q(B)$ on $L_6^2$ are given in the following table:
\begin{table}[h]
\centering
 { 
\begin{tabular}{c|c|c}
rectangle $R$ & $V_{Q(A)}(R)$ & $V_{Q(B)}(R)$ \\
\hline\hline
$[0,2]\times [0,3] $ & $ 1$ & $ 1$ \\
$[1,5]\times [1,3] $ & $ 1$ & $ 2$ \\
$[2,4]\times [2,4] $ & $ -2$ & $ -1$ \\
$[2,5]\times [2,4] $ & $ -1$ & $ 0$ \\
$[2,5]\times [3,5] $ & $ 1$ & $ 0$ \\
$[2,5]\times [1,4] $ & $ 0$ & $ 1$ \\
\end{tabular}
}
\end{table}
\end{example}
}

 {We now introduce the notion of an imprecise discrete copula \cite{KoPe-DIC}.}

\begin{definition}
    A pair $(P,Q)$ of functions $P,Q:L_n^2\to {[0,n]}$ is called a \emph{discrete imprecise copula} if $P$ and $Q$ satisfy the following conditions:
    \begin{enumerate}
        \item[(IC1)]  Both $P$ and $Q$ satisfy condition $(Q1')$.
        \item[(IC2)] For each rectangle $R=[i,j]\times [k,l]\in\R$ we have:
        \begin{align*}
            Q(i,k)+P(j,l)-P(i,l)-P(j,k)&\geq 0,\\
            P(i,k)+Q(j,l)-P(i,l)-P(j,k)&\geq 0,\\
            Q(i,k)+Q(j,l)-Q(i,l)-P(j,k)&\geq 0,\\
            Q(i,k)+Q(j,l)-P(i,l)-Q(j,k)&\geq 0.
        \end{align*}
    \end{enumerate}
\end{definition}
It follows from a result of Montes et al. \cite[Prop. 2]{MMPV} that, given a {discrete} imprecise copula $(P,Q)$, both $P$ and $Q$ are quasi-copulas and $P(r,s)\leq Q(r,s)$ for all $(r,s)$, i.e., $P\leq Q$ in the {point}-wise order.
Imprecise discrete copula $(P,Q)$ can be interpreted as a set of all discrete copulas bounded by $P$ and $Q$, i.e., discrete copulas $C$ such that $P \le C \le Q$. As such it can model the imprecision in the knowledge of a single exact discrete copula.
In this interpretation, two important properties of imprecise copulas arise: the avoidance of sure loss and coherence.

\begin{definition}\label{def_no sure loss}
    A discrete imprecise copula {$(P,Q)$}  {\emph{defined on $L_n^2$ avoids sure loss}} if there exists a discrete copula $C$ defined on $L_n^2$ such that $P\leq C\leq Q$. 
\end{definition}

Avoidance of sure loss assures that the imprecise model is not vacant. Coherence, on the other hand, assures that the quasi-copulas $P$ and $Q$ are tight bounds for the imprecise model.

\begin{definition}\label{def_coherence}
    Suppose that a {discrete} imprecise copula $(P,Q)$  {\emph{avoids sure loss}}, i.e., the set $\C(P,Q)$ of all copulas such that $P\leq C\leq Q$ is not empty. Then $(P,Q)$ is \emph{coherent} if $P(r,s)=\inf_{C\in\C(P,Q)} C(r,s)$ and $Q(r,s)=\sup_{C\in\C(P,Q)} C(r,s)$ for all $r,s\in L_n$. 
\end{definition}

All the definitions above can be considered also for function on the full domain. One simply has to replace the set $I_n^2$ (respectively $L_n^2$) with the unit square $[0,1]^2$. We write down the definitions of full domain quasi-copulas and imprecise copulas explicitly by keeping the notation of the conditions aligned with the discrete cases, other definitions are extended accordingly.

 {A function $Q:[0,1]^2 \rightarrow [0,1]$ is \emph{a quasi-copula} if and only if it satisfies the following conditions:
\begin{enumerate}
    \item[(Q1)] $Q(x,0)=Q(0,y)=0$ and $Q(x,1)=x$, $Q(1,y)=y$ for every $x,y\in [0,1]$,
    \item[(Q2)] $Q$ is supermodular on any rectangle with at least one edge on the boundary of the unit square, i.e., 
    \[Q(x_1,y_1) + Q(x_2,y_2) \geq Q(x_1,y_2) + Q(x_2,y_1)\]
    whenever $x_1\leq x_2$, $y_1\leq y_2$ and at least one of $x_1$, $x_2$, $y_1$, $y_2$ is equal to either $0$ or $1$. 
\end{enumerate}}

{ { Equivalently, the authors of \cite{genest_characterization_1999} characterize quasi-copulas as functions that satisfy the boundary condition (Q1) and}}

\begin{enumerate}
     \item[(Q$2^{\dagger}$)]  {$Q$ is increasing in each component.}
	\vspace{3pt}
     \item[(Q$3^{\dagger}$)]  $Q$ satisfies the 1-Lipschitz condition, i.e., for every $x_1,x_2,y_1,y_2 \in [0,1]$ 
    $$ |Q(x_2,y_2) -Q(x_1,y_1)| \leq |x_2 - x_1| + |y_2 - y_1|.$$
\end{enumerate}

 {Given a quasi-copula $Q$, we define the $Q$-volume of a rectangle $R=[x_1,x_2]\times[y_1,y_2]\subset [0,1]^2$ with $x_1\leq x_2$, $y_1\leq y_2$ by 
\begin{equation}\label{V_Q}
    V_Q(R)=Q(x_1,y_1) + Q(x_2,y_2) - Q(x_1,y_2) - Q(x_2,y_1).
\end{equation}}
We remark that any discrete quasi-copula defined on $I_n^2$ can be extended to a full domain quasi-copula, so discrete quasi-copulas can equivalently be defined as restriction of full domain quasi-copulas to $I_n^2$.

 {The following definition recalls the notion of an imprecise copula on the full domain originally introduced in \cite{MMPV}.}

 {
\begin{definition}
    A pair $(P,Q)$ of functions $P,Q:[0,1]^2\to {[0,1]}$ is called an \emph{imprecise copula} if $P$ and $Q$ satisfy the following conditions:
    \begin{enumerate}
        \item[(IC1)]  Both $P$ and $Q$ satisfy condition $(Q1)$.
        \item[(IC2)] For any $x,x',y,y'\in [0,1]$, with $x\le x'$ and $y\le y'$, we have:
        \begin{align*}
            Q(x',y')+P(x,y)-P(x,y')-P(x',y)&\geq 0,\\
            P(x',y')+Q(x,y)-P(x,y')-P(x',y)&\geq 0,\\
            Q(x',y')+Q(x,y)-Q(x',y)-P(x,y')&\geq 0,\\
            Q(x',y')+Q(x,y)-P(x',y)-Q(x,y')&\geq 0.
        \end{align*}
    \end{enumerate}
\end{definition}
}

\section{Defect based transformations of minimal range discrete quasi-copulas}
\label{sec:defect_transform}

\medskip
The extreme points of $\A^B_n$ and $\B_n$ play a special role in their (quasi-) copula counterparts. 
In particular, they are linked with discrete quasi-copulas and discrete copulas whose range is exactly equal to $L_n$. Since $L_n$ is always contained in the range of $Q \in \Q_n$, we call discrete quasi-copulas $Q$ with $\im(Q)=L_n$ \emph{of minimal range} (analogously for discrete copulas). Such discrete quasi-copulas are called irreducible in \cite[Def. 8]{AST-1} or \cite[Def. 2.3]{Kob} because they are in fact the extreme points of $\Q_n$. We denote the set of all quasi-copulas on $L_n$ of minimal range by $\cE_n$. Quasi-copulas of minimal range will be the main focus of our paper. Observe that they are exactly the corner sum matrices of alternating sign matrices \cite{BiKo,BS,Propp,RoRu}. So the bijective correspondence between $\Q_n$ and $\A^B_n$ restricts to the correspondence between $\cE_n$ and the set of all ASMs.

We are also interested in the existence of a discrete copula $C$ such that $Q_1\leq C\leq Q_2$ for two quasi-copulas $Q_1$ and $Q_2$ of minimal range. To address this, we need to introduce the notions of \textit{defect} and of \textit{imprecise copula} in the discrete setting.

Defects of quasi-copulas were introduced in Dibala et al. \cite[\S 3]{DS-PMK} in general setting and adjusted to the discrete setting in \cite{KoPe-DIC}. Since we prefer the usual directions for the numbering of rows and columns in matrices, all our arrows in the notation are rotated by $90^{\circ}$ as compared to the arrows in \cite{DS-PMK}. So the arrows correspond to directions in the matrices rather than directions in the coordinate system in the domain of  {quasi-}copulas.

Recall that $\R$ is the set of all rectangles on the grid $L_n^2$. For each point $(r,s)\in L_n^2$ we introduce four subsets of rectangles. These are rectangles that have one of the corners fixed on the grid. In particular, we define:
\begin{align*}
    \R_{\searrow}(r,s)&=\left\{[r,r+i]\times [s,s+j];\ 1\leq i\leq n-r, 1\leq j\leq n-s\right\},\\
    \R_{\swarrow}(r,s)&=\left\{[r,r+i]\times [j,s];\ 1\leq i\leq n-r, 0\leq j\leq s-1\right\},\\
    \R_{\nwarrow}(r,s)&=\left\{[i,r]\times [j,s];\ 0\leq i\leq r-1, 0\leq j\leq s-1\right\},\\
    \R_{\nearrow}(r,s)&=\left\{[i,r]\times [s,s+j];\ 0\leq i\leq r-1, 1\leq j\leq n-s\right\}.
\end{align*}
We illustrate this by an example.

\begin{example}
{Rectangle} { $R=[2,5]\times [1,5]$} shown with red squares in Figure~\ref{fig:rectangle} { belongs to $\R_{\searrow}(2,1)$, $\R_{\swarrow}(2,5)$, $ \R_{\nwarrow}(5,5)$ and  $\R_{\nearrow}(5,1)$. Note that the vertices of the rectangle have coordinates $(2,1)$, $(2,5)$, $(5,5)$ and $(5,1)$ on the grid. Each of these is marked by a small \emph{x} on the figure.}
\begin{figure}[ht]
\begin{center}
\includegraphics[height=7cm]{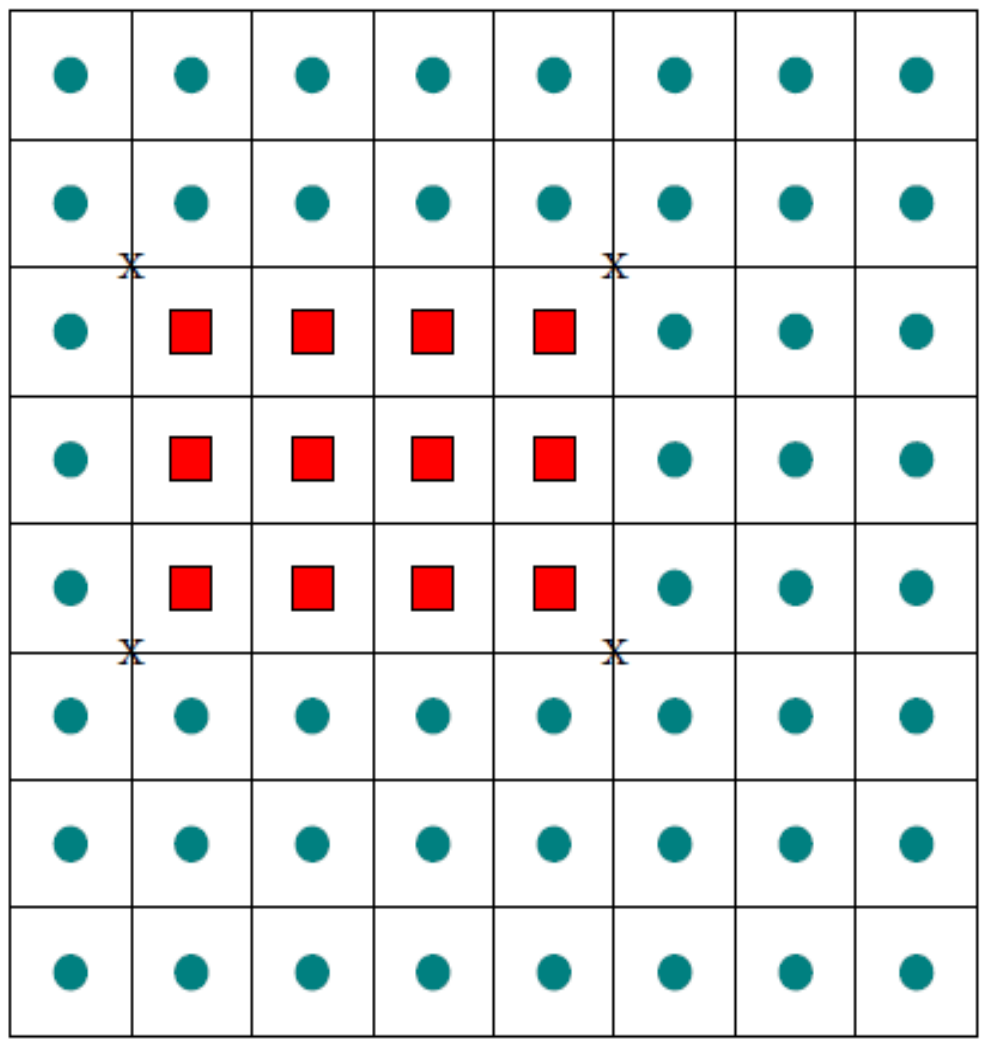}
\caption{A rectangle in an $8 \times 8$ matrix corresponding to the grid $L_8^2$.}
\end{center}
\label{fig:rectangle}
\end{figure}

The $Q$-volume of $R$ is equal to the sum of all entries of $A(Q)$ with indices in $(2,5] \times (1,5]$. These entries represent masses in $1\times 1$ subsquares of $R$.
\end{example}

Now, for each discrete quasi-copula $Q\in\Q_n$ we define four \emph{directional defect matrices} $D_{\nearrow}^Q,D_{\nwarrow}^Q,D_{\swarrow}^Q$ and $D_{\searrow}^Q$.

They are given by
\begin{align*}
    D_{\searrow}^Q(r,s)&=\min\{0,V_Q(R);\ R\in\R_{\searrow}(r,s)\},\\
    D_{\swarrow}^Q(r,s)&=\min\{0,V_Q(R);\ R\in\R_{\swarrow}(r,s)\},\\
    D_{\nwarrow}^Q(r,s)&=\min\{0,V_Q(R);\ R\in\R_{\nwarrow}(r,s)\},\\
    D_{\nearrow}^Q(r,s)&=\min\{0,V_Q(R);\ R\in\R_{\nearrow}(r,s)\}.
\end{align*}
 {We here omit the defect values at $r=0$ and $s=0$, as they are all $0$ since $Q$ is a quasi-copula. Hence, we consider the defect matrices as matrices of size $n \times n$. Their entries are labeled by the set $\{1,2,\ldots,n\} \times \{1,2,\ldots,n\}$.}

Observe that $Q$ is a copula if and only if any (and hence all) of the defect matrices is equal to $0$. So, it is only interesting to study defect matrices for proper quasi-copulas. 

Two additional defect matrices are important. They are called \emph{the main and the opposite defect matrices} and denoted by $D_M^Q$ and $D_O^Q$, respectively. Their entries are given by
\begin{align*}
    D_M^Q(r,s)&=\min\left\{D_{\searrow}^Q(r,s),D_{\nwarrow}^Q(r,s)\right\},\\
    D_O^Q(r,s)&=\min\left\{D_{\swarrow}^Q(r,s),D_{\nearrow}^Q(r,s)\right\}.
\end{align*}
To simplify the notation we replace the superscript $Q$ by $A$ when we consider an ASM matrix $A$ instead of the corresponding quasi-copula $Q(A)$. Then we write $D_{\searrow}^A$ instead of $D_{\searrow}^{Q(A)}$, and similarly for other defects. In such a case, we refer to the defect matrices of an ASM, meaning the defect matrices of the corresponding quasi-copula $Q(A)\in\Q_n$.

\begin{example}\label{Ex:F23+}
    Consider the ASM below $$F_3^2=\begin{pmatrix}
        0 & 1 & 0 \\
        1 & -1 & 1 \\
        0 & 1 & 0
\end{pmatrix}$$ and its associated discrete quasi-copula 
$$Q=Q(F_3^2)=\begin{pmatrix}
        0 & 1 & 1 \\
        1 & 1 & 2 \\
        1 & 2 & 3
\end{pmatrix}.$$ The corresponding directional defect matrices are
$$D_{\searrow}^{Q}=\begin{pmatrix}
       -1 & 0 & 0 \\
       0  & 0 & 0 \\
       0  & 0 & 0
\end{pmatrix},
    D_{\swarrow}^{Q}=\begin{pmatrix}
       0  & -1& 0 \\
       0  & 0 & 0 \\
       0  & 0 & 0
\end{pmatrix}, 
    D_{\nwarrow}^{Q}=\begin{pmatrix}
       0  & 0 & 0 \\
       0  & -1& 0 \\
       0  & 0 & 0
\end{pmatrix}, 
%     \ \mathrm{and}\ 
    D_{\nearrow}^{Q}=\begin{pmatrix}
        0 & 0 & 0 \\
        -1 & 0 & 0 \\
        0 & 0 & 0
\end{pmatrix}
$$
The main and the opposite defect matrices are
$$
    D_M^{Q}=\begin{pmatrix}
       -1  & 0 & 0 \\
       0  & -1& 0 \\
       0  & 0 & 0
\end{pmatrix}, 
     \ \mathit{and}\ 
    D_O^{Q}=\begin{pmatrix}
        0 & -1 & 0 \\
        -1 & 0 & 0 \\
        0 & 0 & 0
\end{pmatrix}.
$$

    Two other matrices and each with one of its defect matrices are
$$
   A= \begin{pmatrix}
                        0 & 0  & 0 & 1 & 0  & 0 \\
                        0  & 0 & 1 & 0 & 0 & 0\\
                        0 & 1 & 0 & -1 & 0 & 1\\
                        1 & 0 & -1 & 0 & 1 & 0\\
                        0 & 0 & 0 & 1 & 0 & 0\\
                        0 & 0 & 1 & 0 & 0 & 0 \\                        
\end{pmatrix}
\textrm{ with } 
D_M^A=  \begin{pmatrix}
                         0 & -1 & -1 & 0 & 0& 0\\
                         -1 & -2 & -1 & 0 & 0& 0\\
                        -1 & -1 & 0  & -1 & -1& 0\\
                         0 & 0 & -1 & -2 &  -1& 0\\
                         0 & 0 & -1 & -1 & 0 & 0\\
                        0 & 0  & 0 & 0 & 0 & 0 \\
\end{pmatrix}
$$
and
$$B=\begin{pmatrix}
0 & 0 & 0 & 1 & 0 & 0 \\
0 & 0 & 1 & 0 & 0 & 0 \\
0 & 1 & 0 & -1 & 1 & 0 \\
1 & 0 & -1 & 1 & -1 & 1 \\
0 & 0 & 1 & 0 & 0 & 0 \\
0 & 0 & 0 & 0 & 1 & 0
\end{pmatrix} 
\textrm{ with }
D^B_{\searrow}= \begin{pmatrix}
0 & 0 & -1 & 0 & 0 & 0 \\
0 & -1 & -1 & 0 & 0 & 0 \\
-1 & -1 & 0 & -1 & 0 & 0 \\
0 & 0 & 0 & 0 & 0 & 0 \\
0 & 0 & 0 & 0 & 0 & 0 \\
0 & 0 & 0 & 0 & 0 & 0
\end{pmatrix}
$$
\end{example}

Using the defect matrices we obtain six transformations on quasi-copulas. They are:
\begin{align*}
    Q_{\searrow}&=Q-D^Q_{\searrow},\\
    Q_{\swarrow}&=Q+D^Q_{\swarrow},\\
    Q_{\nwarrow}&=Q-D^Q_{\nwarrow},\\
    Q_{\nearrow}&=Q+D^Q_{\nearrow},\\
    Q_{M}&=Q-D^Q_{M},\\
    Q_{O}&=Q+D^Q_{O}.
\end{align*}
Dibala et al. \cite[Thm. 4.3]{DS-PMK} show that all six transformations applied to a quasi-copula yield quasi-copulas.
 {Furthermore, they prove in \cite[Thm. 5.2]{DS-PMK} that a pair $(P,Q)$ of quasi-copulas is an imprecise copula if and only if $P_M\le Q$ and $P\le Q_O$. These properties hold also in the discrete setting.}

\begin{theorem}\label{thm:ASM_transformations}
    Suppose that $A$ is a proper ASM and $Q(A)$ the corresponding quasi-copula. Then all six transformations applied to $Q(A)$ yield a quasi-copula such that the corresponding ABM is an ASM.
\end{theorem}

\begin{proof}
    Suppose that $A\in\A_n$ and that $Q(A)$ is the corresponding mapping $L_n\times L_n\to L_n$. Then all six matrices corresponding to defects are matrices with integer entries. Also, the transformed matrices $Q(A)_*$, where $*\in\{\searrow,\swarrow,\nwarrow,\nearrow, M,O\}$ are matrices over integer numbers. By \cite[Thm. 4.3]{DS-PMK}, they are quasi-copulas. By the $1$-Lipschitz property (Q3$^{\dagger} $) it follows that the corresponding ABM has all entries in $\{-1,0,1\}$. Moreover, since it has to satisfy conditions \eqref{ABM_conditions}, it follows that the two consecutive nonzero entries in each row and in each column have to alternate in sign. Namely, we have 
    $$0\leq\sum_{r=1}^{j-1}a_{ir} \leq 1, \quad 0\leq \sum_{r=1}^{k}a_{ir}\leq 1$$ 
    for any $i$ and any $j,k$ with $j<k$ and therefore 
    $$\left|\sum_{r=j}^{k}a_{ir}\right|\leq 1.$$ 
    Then if $a_{ij}$ and $a_{ik}$ are two consecutive nonzero entries in a row they have to be of distinct sign. The argument for column entries is similar.
\end{proof}

\begin{example}
    Consider again the matrix $F_3^2$ from Example \ref{Ex:F23+}. For its associated quasi-copula $Q=Q(F_3^2)$, the six transformations give
$$Q_{\searrow}=\begin{pmatrix}
        1 & 1 & 1 \\
        1 & 1 & 2 \\
        1 & 2 & 3
\end{pmatrix},\ 
    Q_{\swarrow}=\begin{pmatrix}
        0 & 0 & 1 \\
        1 & 1 & 2 \\
        1 & 2 & 3
\end{pmatrix},\  
    Q_{\nwarrow}=\begin{pmatrix}
        0 & 1 & 1 \\
        1 & 2 & 2 \\
        1 & 2 & 3
\end{pmatrix}, $$
%     \ \mathrm{and}\ 
$$    Q_{\nearrow}=\begin{pmatrix}
        0 & 1 & 1 \\
        0 & 1 & 2 \\
        1 & 2 & 3
\end{pmatrix},\ 
    Q_M=\begin{pmatrix}
        1 & 1 & 1 \\
        1 & 2 & 2 \\
        1 & 2 & 3
\end{pmatrix}, 
     \ \mathit{and}\ 
    Q_O=\begin{pmatrix}
        0 & 0 & 1 \\
        0 & 1 & 2 \\
        1 & 2 & 3
\end{pmatrix}.
$$
The ASMs corresponding to these quasi-copulas, that are in fact copulas, are exactly all six $3\times 3$ permutation matrices
$$A\left(Q_{\searrow}\right)=\begin{pmatrix}
        1 & 0 & 0 \\
        0 & 0 & 1 \\
        0 & 1 & 0
\end{pmatrix},\ 
    A\left(Q_{\swarrow}\right)=\begin{pmatrix}
        0 & 0 & 1 \\
        1 & 0 & 0 \\
        0 & 1 & 0
\end{pmatrix},\  
    A\left(Q_{\nwarrow}\right)=\begin{pmatrix}
        0 & 1 & 0 \\
        1 & 0 & 0 \\
        0 & 0 & 1
\end{pmatrix}, $$
%     \ \mathrm{and}\ 
$$    A\left(Q_{\nearrow}\right)=\begin{pmatrix}
        0 & 1 & 0 \\
        0 & 0 & 1 \\
        1 & 0 & 0
\end{pmatrix},\ 
    A\left(Q_M\right)=\begin{pmatrix}
        1 & 0 & 0 \\
        0 & 1 & 0 \\
        0 & 0 & 1
\end{pmatrix}, 
     \ \mathit{and}\ 
    A\left(Q_O\right)=\begin{pmatrix}
        0 & 0 & 1 \\
        0 & 1 & 0 \\
        1 & 0 & 0
\end{pmatrix}.
$$

For the matrices $A$ and $B$ from Example~\ref{Ex:F23+} we have
$$Q(A)=\begin{pmatrix}
0 & 0 & 0 & 1 & 1 & 1 \\
0 & 0 & 1 & 2 & 2 & 2 \\
0 & 1 & 2 & 2 & 2 & 3 \\
1 & 2 & 2 & 2 & 3 & 4 \\
1 & 2 & 2 & 3 & 4 & 5 \\
1 & 2 & 3 & 4 & 5 & 6
\end{pmatrix},\ 
Q(A)_M=\begin{pmatrix}
0 & 1 & 1 & 1 & 1 & 1 \\
1 & 2 & 2 & 2 & 2 & 2 \\
1 & 2 & 2 & 3 & 3 & 3 \\
1 & 2 & 3 & 4 & 4 & 4 \\
1 & 2 & 3 & 4 & 4 & 5 \\
1 & 2 & 3 & 4 & 5 & 6
\end{pmatrix}
$$
and
$$Q(B)=\begin{pmatrix}
0 & 0 & 0 & 1 & 1 & 1 \\
0 & 0 & 1 & 2 & 2 & 2 \\
0 & 1 & 2 & 2 & 3 & 3 \\
1 & 2 & 2 & 3 & 3 & 4 \\
1 & 2 & 3 & 4 & 4 & 5 \\
1 & 2 & 3 & 4 & 5 & 6
\end{pmatrix},\quad Q(B)_{\searrow}=\begin{pmatrix}
0 & 0 & 1 & 1 & 1 & 1 \\
0 & 1 & 2 & 2 & 2 & 2 \\
1 & 2 & 2 & 3 & 3 & 3 \\
1 & 2 & 2 & 3 & 3 & 4 \\
1 & 2 & 3 & 4 & 4 & 5 \\
1 & 2 & 3 & 4 & 5 & 6
\end{pmatrix}. $$
The ASMs corresponding to the two transformation are 
$$A\left(Q(A)_M\right)=\begin{pmatrix}
0 & 1 & 0 & 0 & 0 & 0 \\
1 & 0 & 0 & 0 & 0 & 0 \\
0 & 0 & 0 & 1 & 0 & 0 \\
0 & 0 & 1 & 0 & 0 & 0 \\
0 & 0 & 0 & 0 & 0 & 1 \\
0 & 0 & 0 & 0 & 1 & 0
\end{pmatrix} \quad\text{ and } \quad A\left(Q(B)_{\searrow}\right)=\begin{pmatrix}
0 & 0 & 1 & 0 & 0 & 0 \\
0 & 1 & 0 & 0 & 0 & 0 \\
1 & 0 & -1 & 1 & 0 & 0 \\
0 & 0 & 0 & 0 & 0 & 1 \\
0 & 0 & 1 & 0 & 0 & 0 \\
0 & 0 & 0 & 0 & 1 & 0
\end{pmatrix}. $$
\end{example}

\section{ASM imprecise copulas corresponding to dense ASMs are always coherent}\label{five}

We now further restrict to the case of dense ASMs and analyze the properties of the associated ASM imprecise copulas. 
An ASM is \emph{dense} if in each row and in each column there is no zero entries between two nonzero entries. The structure theorem for such matrices is proved in \cite[Thm. 4.4]{KoPe-DIC}. Basic building blocks of dense ASMs are matrices $F_n^k$ introduced by Brualdi and Schroeder \cite{BS}. For $k=1,2,\ldots,n$, $n\ge 1$, $F_n^k$ is a dense alternating sign matrix which has all entries in four diagonal stripes $(k,1),(k-1,2)$,$\ldots, (1,k)$ and $(1,k),(2,k+1),\ldots, (n-k+1,n)$ and $(n-k+1,n),(n-k+2,n-1),\ldots,(n,n-k+1)$ and $(n,n-k+1),(n-1,n-k),\ldots, (k,1)$ equal to $1$. Other entries between these stripes are nonzero by the denseness condition. In particular, $F_n^1$ is equal to the identity matrix 
$$\begin{pmatrix}
  1 & 0 & 0 & \cdots  & 0 & 0\\
  0 & 1 & 0 & \ddots  & 0 & 0\\
  0 & 0 & 1 &\ddots  & 0 & 0\\
 \vdots & \ddots & & \ddots & \ddots & \vdots\\
 0 & 0 & 0&  \ddots & 1 & 0\\
 0 & 0 & 0&  \cdots & 0 & 1
\end{pmatrix},$$ 
and $F_n^n$ is the antidiagonal matrix 
$$\begin{pmatrix}
  0 & 0 &  \cdots  & 0 & 0 & 1\\
  0 & 0 &  \iddots  & 0 & 1 & 0\\
  0 & 0 & \iddots  & 1 & 0 & 0\\
  \vdots  & \iddots & \iddots & \iddots  & & \vdots\\
  0 & 1 &  \iddots & 0 & 0 & 0 \\
  1 & 0 &   \cdots & 0 & 0 &0
\end{pmatrix}.$$ 
These two matrices correspond to the Fr\'echet upper and lower bound, respectively. All other matrices $F_n^k$ are proper ASMs. For instance, when $n=5$ or $n=6$ we have
$$F_5^2=\begin{pmatrix}
    0 & 1 & 0 & 0 & 0 \\
    1 & -1 & 1 & 0 & 0\\
    0 & 1 & -1 & 1 & 0\\
    0 & 0 & 1 & -1 &  1\\
    0 & 0 & 0 & 1 & 0
\end{pmatrix}, 
F_5^3=\begin{pmatrix}
    0 & 0 & 1 & 0 & 0 \\
    0 & 1 & -1 & 1 & 0\\
    1 & -1 & 1 & -1 & 1\\
    0 & 1 & -1 & 1 & 0 \\
    0 & 0 & 1 & 0 & 0
\end{pmatrix},$$
and
$$F_6^4=
\begin{pmatrix}
    0 & 0 & 0 & 1 & 0 & 0\\
    0 & 0 & 1 &-1 & 1 & 0\\
    0 & 1 &-1 & 1 &-1 & 1\\
    1 &-1 & 1 &-1 & 1 & 0\\
    0 & 1 &-1 & 1 & 0 & 0\\
    0 & 0 & 1 & 0 & 0 & 0\\
\end{pmatrix}. $$

A dense ASM is \emph{irreducible} if it is either equal to $1\times 1$ matrix $1$ or it is equal to $F_n^k$ for some $n\ge 3$ and $2\le k\le n-1$.

 {Since matrices $F_n^k$ are symmetric and might be invariant also under other symmetry transformation, also the corresponding quasi-copulas inherit the same symmetries. By \cite[Thm. 4.4]{KoPe-DIC} any quasi-copula corresponding to a dense ASM is up to a permutation an ordinal sum of quasi-copulas corresponding to matrices $F_n^k$. (For the definition and some properties of ordinal sums we refer to \cite[\S{5.2}]{Schweizer-Sklar} and to \cite{DuKlSaPlSe}.) Thus the structure of a quasi-copula corresponding to a dense ASM is completely described by the structure of quasi-copulas $Q(F_n^k)$ and the corresponding permutation.}

\smallskip

The following theorem follows from the results proved in \cite{KoPe-DIC}.

\begin{theorem}\label{thm:dense_implies_coherent}
    Suppose that $(P,Q)$ is an imprecise copula such that both $P$ and $Q$ correspond to dense ASMs. Then  $(P,Q)$  is coherent. In particular, $(P,Q)$ avoids sure loss.
\end{theorem}

\begin{proof}
First, we assume that the ASM $A$ corresponding to $P$ is an irreducible dense ASM. Then $A=F_n^k$ for some $n\ge 3$ and $2\le k\le n-1$ since we assumed that $n\ge 2$. By \cite[Thm. 4.8]{KoPe-DIC}, $P_M=Q\left(F_n^{k-1}\right)$. Consider first the case $k\ge 3$. Then \cite[Thm. 4.10]{KoPe-DIC} implies that $P=\min_{C\in\C(P,P_M)} C $. Since $(P,Q)$ is an imprecise copula, \cite[Thm. 5.2]{DS-PMK} yields $P_M\le Q$. Hence, $\C(P,P_M)\subseteq\C(P,Q)$ and therefore $P=\min_{C\in\C(P,Q)} C $. 

Now suppose that $k=2$. Then, $P_M=Q(F_n^1)=M$ is the Fr\'echet-Hoeffding upper bound. This implies that $Q=M$ since $P_M\le Q$ by \cite[Thm. 5.2]{DS-PMK}. Arguing as in the proof of \cite[Theorem~4.10]{KoPe-DIC}, we obtain $ P=\min_{C\in\C(P,Q)} C.$

Next, suppose that the ASM $B$ corresponding to $Q$ is irreducible and dense. By arguments analogous to those above, we obtain that $P\le Q_O$ and $Q=\max_{C\in\C(Q_O,Q)} C=\max_{C\in\C(P,Q)} C$. 

Assume now that neither of the ASM corresponding to $P$ and $Q$ is irreducible. 
    By \cite[Thm. 4.4]{KoPe-DIC} the ASM $A$ corresponding to $P$ has block structure form 
    $$A=\begin{pmatrix}
 A_{11} & A_{12} & \cdots  & A_{1l} \\
 A_{21} & A_{22} & \cdots  & A_{2l} \\
 \vdots & \vdots & & \vdots\\
 A_{l1} & A_{l2} & \cdots  & A_{ll} 
\end{pmatrix},$$
    where $l\ge 2$ and there exists a permutation $\sigma$ of $\{1,2,\ldots,l\}$ such that block $A_{r,\sigma(r)}$  is an irreducible dense ASM  and $A_{rs}=0$ for $s\neq \sigma(r)$.
    Hence, each nonzero block that is not $1\times 1$ is of the form $F_{m_r}^{k_r}$, where integers $m_r$ and $k_r$ are such that $m_r\ge 3$ and $2\le k_r\le m_r-1$.  Since $P$ is a proper quasi-copula, we have $m_r\ge 3$ for at least one index $r$. 
    
    All nonzero blocks of $A$ correspond to ASM quasi-copulas. By Property (Q2) of a quasi-copula it follows that the $P$-volume of any rectangle with at least one edge on the boundary of the grid corresponding to the block structure of $A$ is nonnegative. Also the $P$-volume of any rectangle that is a union of such rectangles is nonnegative. Hence, if $V_Q(R)<0$ then $R$ is contained in a single block of $A$. Consequently, the defect matrix $D_M^P$ has a block structure compatible with that of $A$, and its nonzero blocks occur only in positions $(r,\sigma(r))$ with $m_r\ge 3$ and correspond to main defect matrices of irreducible dense ASMs $F_{m_r}^{k_r}$.    
    Hence, the imprecise copula $(P,P_M)$ is a patchwork of quasi-copulas corresponding to $(F_{m_r}^{k_r},F_{m_r}^{k_r-1})$ (cf. \cite[Thm. 4.8]{KoPe-DIC} ), where $m_r\ge 3$, and constant blocks when $s\neq\sigma(r)$ or $m_r=1$. Since for each block  $(Q(F_{m_r}^{k_r}),Q(F_{m_r}^{k_r-1}))$ the first quasi-copula corresponds to an irreducible dense ASM and is coherent, it follows that $P=\min_{C\in\C(P,Q)} C$. 

    Finally, analogous arguments show that $Q=\max_{C\in\C(P,Q)} C $. Therefore, $(P,Q)$ is coherent.
\end{proof}

The following example shows that the block structure for $P$ and $Q$ of an imprecise copula $(P,Q)$ need not coincide.

\begin{example}
    Consider two $5\times 5$ ASMs with block structures
    $$A=F_5^3 \quad \text{and}\quad B=\begin{pmatrix}
F_4^2 & 0  \\
 0 & 1 
\end{pmatrix}.$$
Their corresponding discrete quasi-copulas are 
    $$P=Q(A)=\begin{pmatrix}
 0 & 0 & 1 & 1 & 1\\
 0 & 1 & 1 & 2 & 2\\
 1 & 1 & 2 & 2 & 3\\
 1 & 2 & 2 & 3 & 4 \\
 1 & 2 & 3 & 4 & 5 
\end{pmatrix}\quad \text{and}\quad Q=Q(B)=\begin{pmatrix} 
 0 & 1 & 1 & 1 & 1\\
 1 & 1 & 2 & 2 & 2\\
 1 & 2 & 2 & 3 & 3\\
 1 & 2 & 3 & 4 & 4 \\
 1 & 2 & 3 & 4 & 5 
\end{pmatrix}.$$
The quasi-copulas $P_M$ and $Q_O$ correspond, respectively, to the ASMs
$$F_5^2\quad \text{and} \quad \begin{pmatrix}
F_4^3 & 0  \\
 0 & 1 
\end{pmatrix}.$$
They are given by
$$P_M=\begin{pmatrix}
 0 & 1 & 1 & 1 & 1\\
 1 & 1 & 2 & 2 & 2\\
 1 & 2 & 2 & 3 & 3\\
 1 & 2 & 3 & 3 & 4 \\
 1 & 2 & 3 & 4 & 5 
\end{pmatrix}
\quad \text{and}\quad Q_O=\begin{pmatrix} 
 0 & 0 & 1 & 1 & 1\\
 0 & 1 & 1 & 2 & 2\\
 1 & 1 & 2 & 3 & 3\\
 1 & 2 & 3 & 4 & 4 \\
 1 & 2 & 3 & 4 & 5 
\end{pmatrix}.$$
Comparing the quasi-copulas component-wise, we obtain $P_M\le Q$ and $P\le Q_O$. Hence, $(P,Q)$ is an imprecise copula.
\end{example}

Theorem~\ref{thm:ASM_transformations} describes the behavior of general ASMs under all six defect transformations. Next proposition and example describe the behavior of dense AMSs with respect to the same transformations.

 { 
\begin{proposition}\label{dense_invariant}
    Suppose that $Q$ is a quasi-copula that corresponds to a dense ASM. Then its main $Q_M$ and opposite $Q_O$ defect transformations also correspond to dense ASMs.
\end{proposition}
 }

 {
\begin{proof}
    Suppose that $A$ is the dense ASM corresponding to $Q$. By \cite[Thm. 4.4]{KoPe-DIC} $A$ is up to a permutation an ordinal sum of quasi-copulas corresponding to irreducible dense ASMs of the form $F_{n_r}^{k_r}$ for some integers $n_r$ and $k_r$, where either $m_r=k_r=1$ or $n_r\ge 3$ and $2\le k_r\le n_r-1$.
    Arguing simlarly as in the proof of Theorem~\ref{thm:dense_implies_coherent},
    the defect matrices of $Q$ have the same block structure as $A$; they are a block-wise permutation of a block diagonal matrix with the diagonal blocks corresponding to defect matrices of irreducible dense ASMs $F_{n_r}^{k_r}$. By \cite[Thm. 4.8]{KoPe-DIC} the main and the opposite transformations, respectively, of $Q\left(F_{n_r}^{k_r}\right)$ are again quasi-copulas corresponding to $F_{n_r}^{k_r-1}$ and $F_{n_r}^{k_r+1}$, respectively. Thus all the nonzero blocks of the transformed quasi-copulas correspond to dense ASMs. These blocks are irreducible unless $k_r=2$, then $F_{n_r}^1$ is the identity matrix, or $k_r=n_r-1$, then $F_{n_r}^{n_r}$ is the anti-diagonal matrix with the only nonzero entries equal to $1$ on the opposite diagonal. In any case, the transformed quasi-copula corresponds to a dense ASM as well.
\end{proof}
}

 {By Proposition \ref{dense_invariant} it follows that the subclass of quasi-copulas corresponding to dense ASMs is invariant for the main and the opposite transformations. However, the following example shows that it is not invariant for the other (directional) defect transformation.}

\begin{example}
  {    Consider the matrix $F_4^3$. The directional defect matrices of the corresponding quasi-copula $P=Q(F_4^3)$ are 
    $$D_{\searrow}^{P}=\begin{pmatrix}
       0 & -1 & 0 & 0 \\
      -1 & 0  & 0 & 0 \\
       0 & 0  & 0 & 0 \\
       0 & 0  & 0 & 0
\end{pmatrix} \textrm{ and }
    D_{\nwarrow}^{P}=\begin{pmatrix}
       0  &  0 & 0 & 0 \\
       0  &  0 & -1& 0 \\
       0  & -1 & 0 & 0 \\
       0  & 0 & 0 & 0
\end{pmatrix}. 
$$
The ASMs associated with the transformed quasi-copulas are
$$A\left(P_{\searrow}\right)=\begin{pmatrix}
        0 & 1 & 0 & 0 \\
        1 & -1 & 0 & 1 \\
        0 &  0 & 1 & 0 \\
        0 & 1  & 0 & 0
\end{pmatrix}\ \textrm{ and }\ 
    A\left(P_{\nwarrow}\right)=\begin{pmatrix}
        0 & 0 & 1 & 0 \\
        0 & 1 & 0 & 0 \\
        1 & 0 &-1 & 1  \\
        0 & 0 & 1 & 0
\end{pmatrix}.  
$$
These are not dense ASMs.}

 { Consider the next matrix $F_4^2$. The directional defect matrices of the corresponding quasi-copula $Q=Q(F_4^2)$ are 
    $$D_{\swarrow}^{Q}=\begin{pmatrix}
       0 & -1 & 0 & 0 \\
       0 & 0 & -1 & 0 \\
       0 & 0  & 0 & 0 \\
       0 & 0  & 0 & 0
\end{pmatrix} \textrm{ and }
    D_{\nearrow}^{Q}=\begin{pmatrix}
       0  &  0 & 0 & 0 \\
       -1 &  0 & 0 & 0 \\
       0  & -1 & 0 & 0 \\
       0  & 0 & 0 & 0
\end{pmatrix}. 
$$
The ASMs associated with the transformed quasi-copulas are
$$A\left(Q_{\swarrow}\right)=\begin{pmatrix}
        0 & 0 & 1 & 0 \\
        1 & 0 & -1 & 1 \\
        0 & 1 & 0 & 0 \\
        0 & 0 & 1 & 0
\end{pmatrix}\ \textrm{ and }\ 
    A\left(Q_{\nearrow}\right)=\begin{pmatrix}
        0 & 1 & 0 & 0 \\
        0 & 0 & 1 & 0 \\
        1 & -1 & 0 & 1  \\
        0 & 1 & 0 & 0
\end{pmatrix}.  
$$
These are not dense ASMs, as well.}
\end{example}

{ { 
\section{Extensions  {to full domain} quasi-copulas and their defects}\label{sect:extensions}}

In this section we study the extensions of discrete quasi-copulas to the full domain quasi-copulas. Piecewise bilinear extensions are most often considered in the literature. However, for our purpose two other type of extension, that we introduce below, will be more useful.
These are the MW-extension and the  {WM-extension. Since we primarily work with discrete quasi-copulas as matrices, we preserve the usual directions of the rows and columns in the matrices also in their full domain extensions. The coordinate system, we thus obtain, corresponds to the coordinate system $(x,y)$ that is rotated by $90^{\circ}$ in the negative direction compared to the usual coordinate system (see Figure~\ref{fig:extensions}).}

\begin{definition}[{MW-extension of a discrete quasi-copula}]
\label{def:MW}
 {Let $Q$ be a discrete quasi-copula defined on $I_n^2$. 
The \emph{MW-extension} of $Q$ is the quasi-copula $Q^{MW} : [0,1]^2 \to [0,1]$ obtained by extending $Q$ on each square
\[
{R}_{ij}=\left[\frac{i-1}{n},\frac{i}{n}\right]\times\left[\frac{j-1}{n},\frac{j}{n}\right],
\quad i,j=1,\dots,n,
\]
according to the signed mass of $Q$ on ${R}_{ij}$ as follows:
\begin{enumerate}[(i)]
\item if the mass of $Q$ on ${R}_{ij}$ is positive, it is distributed uniformly along the main diagonal of ${R}_{ij}$ (i.e., according to the copula $M(u, v)=\min\{u, v\}$);
\item if the mass is negative, it is distributed uniformly along the counter-diagonal of ${R}_{ij}$ (i.e., according to the copula $W(u,v)=\max\{u+v-1, 0\}$);
\item if the mass is zero, no mass is assigned on ${R}_{ij}$.
\end{enumerate}}
\end{definition}

\begin{definition}[ {WM-extension of a discrete quasi-copula}]
\label{def:WM}
 {Let $Q$ be a discrete quasi-copula defined on $I_n^2$. 
The \emph{WM-extension} of $Q$ is the quasi-copula $Q^{WM} : [0,1]^2 \to [0,1]$ obtained by extending $Q$ on each square
\[
{R}_{ij}=\left[\frac{i-1}{n},\frac{i}{n}\right]\times\left[\frac{j-1}{n},\frac{j}{n}\right],
\quad i,j=1,\dots,n,
\]
according to the signed mass of $Q$ on ${R}_{ij}$ as follows:}
\begin{enumerate}[(i)]
\item  { if the mass of $Q$ on ${R}_{ij}$ is positive, it is distributed uniformly along the counter-diagonal of ${R}_{ij}$ (i.e., according to the copula $W(u, v)=\max\{0,u+v-1\}$); }
\item  { if the mass is negative, it is distributed uniformly along the main diagonal of ${R}_{ij}$ (i.e., according to the copula $M(u,v)=\min\{u,v\}$); }
\item  { if the mass is zero, no mass is assigned on ${R}_{ij}$.}
\end{enumerate}
\end{definition}

 {When working with extensions we use the correspondence between matrices in $\A_n^B$ for various dimensions $n$ and full domain quasi-copulas. This is in fact two step bijection between $\A_n^B$ and $\Q'_n$ via $\Q_n$. Therefore the values of the entries in $A\in\A_n^B$ divided by $n$ represent actual values of volumes for the corresponding full domain quasi-copula. We illustrate this in the following example.}

\begin{example}\label{ex:ext}
 {Consider the ASM
$$F_3^2=\begin{pmatrix}
        0 & 1 & 0 \\
        1 & -1 & 1 \\
        0 & 1 & 0
\end{pmatrix}.$$ }
\vskip 6pt
\begin{figure}[h]
\begin{center}
 \begin{tikzpicture}[scale=5]
    % Outer square
    \draw[thick] (0,0) rectangle (1,-1);

    % Black diagonals on the 4 "central boundary" squares
    % top-middle
    \draw[thick] (1/3,0) -- (2/3,-1/3);
    % bottom-middle
    \draw[thick] (1/3,-2/3) -- (2/3,-1);
    % left-middle
    \draw[thick] (0,-1/3) -- (1/3,-2/3);
    % right-middle
    \draw[thick] (2/3,-1/3) -- (1,-2/3);

    % Red antidiagonal on the central square
    \draw[thick,red] (2/3,-1/3) -- (1/3,-2/3);

    % Axis labels along the top edge (x-axis)
    \node[above] at (0,0)   {$0$};
    \node[above] at (1/3,0) {$\frac{1}{3}$};
    \node[above] at (2/3,0) {$\frac{2}{3}$};
    \node[above] at (1,0)   {$1$};

    % Axis labels along the left edge (y-axis, increasing downward)
    \node[left] at (0,0)    {$0$};
    \node[left] at (0,-1/3) {$\frac{1}{3}$};
    \node[left] at (0,-2/3) {$\frac{2}{3}$};
    \node[left] at (0,-1)   {$1$};

    % Small tick marks at 1/3 and 2/3 on both axes (since the grid lines are gone)
    \draw (1/3,0) -- ++(0,0.02);
    \draw (2/3,0) -- ++(0,0.02);
    \draw (0,-1/3) -- ++(-0.02,0);
    \draw (0,-2/3) -- ++(-0.02,0);
\end{tikzpicture}
\quad\quad     
\begin{tikzpicture}[scale=5]
    % Outer square
    \draw[thick] (0,0) rectangle (1,-1);

    % Black antidiagonals on the 4 "central boundary" squares (reversed)
    % top-middle
    \draw[thick] (2/3,0) -- (1/3,-1/3);
    % bottom-middle
    \draw[thick] (2/3,-2/3) -- (1/3,-1);
    % left-middle
    \draw[thick] (1/3,-1/3) -- (0,-2/3);
    % right-middle
    \draw[thick] (1,-1/3) -- (2/3,-2/3);

    % Red main diagonal on the central square (reversed)
    \draw[thick,red] (1/3,-1/3) -- (2/3,-2/3);

    % Axis labels along the top edge (x-axis)
    \node[above] at (0,0)   {$0$};
    \node[above] at (1/3,0) {$\frac{1}{3}$};
    \node[above] at (2/3,0) {$\frac{2}{3}$};
    \node[above] at (1,0)   {$1$};

    % Axis labels along the left edge (y-axis, increasing downward)
    \node[left] at (0,0)    {$0$};
    \node[left] at (0,-1/3) {$\frac{1}{3}$};
    \node[left] at (0,-2/3) {$\frac{2}{3}$};
    \node[left] at (0,-1)   {$1$};

    % Small tick marks at 1/3 and 2/3 on both axes
    \draw (1/3,0) -- ++(0,0.02);
    \draw (2/3,0) -- ++(0,0.02);
    \draw (0,-1/3) -- ++(-0.02,0);
    \draw (0,-2/3) -- ++(-0.02,0);
\end{tikzpicture}
\end{center}
 { 
\caption{The densities of the MW-extension (left) and the WM-extension (right) of the discrete quasi-copula corresponding to the matrix $F_3^2$ from Example~\ref{ex:ext}. Negative mass is drawn in red. \label{fig:extensions}} }
\end{figure}
 {The densities of the MW and the WM extensions of the corresponding discrete quasi-copula on $I_n^2$ are drawn in Figure \ref{fig:extensions}. Observe that we keep the origin of the coordinate system on the top left corner of the unit square. The $Q^{MW}$ or $Q^{WM}$-volume of each square $\left[\frac{i-1}{3},\frac{i}{3}\right]\times\left[\frac{j-1}{3},\frac{j}{3}\right]$, for $i,j\in\{1,2,3\}$, is in the set $\left\{-\frac13,0,\frac13\right\}$, depending on the value of the corresponding entry in $F_3^2$.}
\end{example}

 {The WM-extension is related to the MW-extension via the flipping operation. Let us describe this operation in detail.  }

 {Suppose that $A\in\A_n^B$ and denote by $\As$ the matrix obtained from $A$ by reversing the order of the columns of $A$. So, if $a_{ij}$ is the $(i,j)$-th entry of $A$ then $a_{i,n-j+1}$ is the $(i,j)$-entry of $\As$. Observe that the operation is such that $\As\in\A_n^B$ and $(\As)^{\sigma}=A$. }

 {For a given quasi-copula $Q:[0,1]^2\to [0,1]$ we write $Q^{\sigma}(x,y)=x-Q(x,1-y)$. It follows from the definition of a quasi-copula that $Q^{\sigma}$ is a quasi-copula as well. Note that $M^{\sigma}=W$ and $W^{\sigma}=M$. This transformation is used to prove the following result. Its content can be observed also on Figure \ref{fig:extensions}: 
By reflecting the left-hand side picture along the vertical symmetry line of the square we obtain the right-hand side picture of the densities. This holds since $F_3^2=\left(F_3^2\right)^{\sigma}$.}

\begin{lemma}\label{lem:flip}
     {Suppose that $A\in\A_n^B$ and consider the MW-extension of the corresponding quasi-copula $Q(A)$. Then 
    $$\left(Q(A)^{MW}\right)^{\sigma}=Q(\As)^{WM}.$$
    Also, if $Q(A)^{WM}$ is the WM-extension of the discrete quasi-copula $Q(A)$ then
    $$\left(Q(A)^{WM}\right)^{\sigma}=Q\left(A^{\sigma}\right)^{MW}.$$}
\end{lemma}

\begin{proof}
     {Suppose that $(x,y)\in R_{ij},$ where $R_{ij}=\left[\frac{i-1}{n},\frac{i}{n}\right]\times \left[\frac{j-1}{n},\frac{j}{n}\right]$ and $1\le i,j\le n$. Let $A\in\A_n^B$ be as in \eqref{matrix A}. If $a_{ij}\ge 0$ then the MW-extension of $Q(A)$ is equal to
\begin{align*}
    Q(A)^{MW}(x,y)&=\frac{1}{n}\sum_{r=1}^{i-1}\sum_{s=1}^{j-1}a_{rs}+\left(\sum_{r=1}^{i-1}a_{rj}\right)\left(y-\frac{j-1}{n} \right)\\ & +\left(\sum_{s=1}^{j-1}a_{is}\right)\left(x-\frac{i-1}{n}\right)+a_{ij}\min\left\{x-\frac{i-1}{n},y-\frac{j-1}{n}\right\},
\end{align*}
and if $a_{ij}\le 0$ then
\begin{align*}
    Q(A)^{MW}(x,y)&=\frac{1}{n}\sum_{r=1}^{i-1}\sum_{s=1}^{j-1}a_{rs}+\left(\sum_{r=1}^{i-1}a_{rj}\right)\left(y-\frac{j-1}{n} \right)\\ &+\left(\sum_{s=1}^{j-1}a_{is}\right)\left(x-\frac{i-1}{n}\right)+a_{ij}\max\left\{0,x+y-\frac{i+j-1}{n}\right\}.
\end{align*}
The WM-extension of $Q(A^{\sigma})$ when $a_{i,n+1-j}\ge 0$ is equal to
\begin{align*}
    Q(A^{\sigma})^{WM}(x,y)&=\frac{1}{n}\sum_{r=1}^{i-1}\sum_{s=n+2-j}^{n}a_{rs}+\left(\sum_{r=1}^{i-1}a_{r,n+1-j}\right)\left(y-\frac{j-1}{n} \right)\\ & +\left(\sum_{s=n+2-j}^{n}a_{is}\right)\left(x-\frac{i-1}{n}\right)+a_{i,n+1-j}\max\left\{0,x+y-\frac{i+j-1}{n}\right\},
\end{align*}
and if $a_{i,n+1-j}\le 0$ then it is equal to
\begin{align*}
    Q(A^{\sigma})^{WM}(x,y)&=\frac{1}{n}\sum_{r=1}^{i-1}\sum_{s=n+2-j}^{n}a_{rs}+\left(\sum_{r=1}^{i-1}a_{r,n+1-j}\right)\left(y-\frac{j-1}{n} \right)\\ & +\left(\sum_{s=n+2-j}^{n}a_{is}\right)\left(x-\frac{i-1}{n}\right)+a_{i,n+1-j}\min\left\{x-\frac{i-1}{n},y-\frac{j-1}{n}\right\}.
\end{align*}
Observe that for $(x,y)\in R_{ij}$ we have $(x,1-y)\in R_{i,n+1-j}$. Assume first that $a_{i,n+1-j}\ge 0$. Then we have
\begin{align*}
    \left(Q(A)^{MW}\right)^{\sigma}(x,y)=&\ x-Q(A)^{MW}(x,1-y)\\
    =&\ x-\frac{1}{n}\sum_{r=1}^{i-1}\sum_{s=1}^{n-j}a_{rs}-\left(\sum_{r=1}^{i-1}a_{r,n+1-j}\right)\left(1-y-\frac{n-j}{n} \right)\\ & -\left(\sum_{s=1}^{n-j}a_{is}\right)\left(x-\frac{i-1}{n}\right)-a_{i,n+1-j}\min\left\{x-\frac{i-1}{n},1-y-\frac{n-j}{n}\right\}
\end{align*}    
We use the equality $\sum_{s=1}^na_{rs}=1$ that holds for each $r$ to obtain
\begin{align*}
    \left(Q(A)^{MW}\right)^{\sigma}(x,y)=&\ x-\frac{i-1}{n}+\frac{1}{n}\sum_{r=1}^{i-1}\sum_{s=n-j+1}^{n}a_{rs}+\left(\sum_{r=1}^{i-1}a_{r,n+1-j}\right)\left(y-\frac{j-1}{n} \right)\\ 
    & -\frac{1}{n}\sum_{r=1}^{i-1}a_{r,n+1-j} -\left(x-\frac{i-1}{n}\right)+\left(\sum_{s=n-j+1}^{n}a_{is}\right)\left(x-\frac{i-1}{n}\right)\\
    & +a_{i,n+1-j}\left(\max\left\{0,x+y-\frac{i+j-1}{n}\right\}-\left(x-\frac{i-1}{n}\right)\right)
\end{align*}    
Finally, cancelation of some of the summands yields
\begin{align*}
    \left(Q(A)^{MW}\right)^{\sigma}(x,y)=&\ \frac{1}{n}\sum_{r=1}^{i-1}\sum_{s=n-j+2}^{n}a_{rs}+\left(\sum_{r=1}^{i-1}a_{r,n+1-j}\right)\left(y-\frac{j-1}{n} \right)\\ & +\left(\sum_{s=n-j+2}^{n}a_{is}\right)\left(x-\frac{i-1}{n}\right)+a_{i,n+1-j}\max\left\{0,x+y-\frac{i+j-1}{n}\right\}\\
    =&\ Q\left(A^{\sigma}\right)^{WM}(x,y).
\end{align*}
When $a_{i,n+1-j}\le 0$ a similar calculation shows that $\left(Q(A)^{MW}\right)^{\sigma}(x,y) =Q\left(A^{\sigma}\right)^{WM}(x,y)$ as well.   }

 { The second part follows since the operation $Q\to Q^{\sigma}$ is an involution. Then the first equality used for $A^{\sigma}$ instead of $A$ implies
$$
\left(Q(A)^{WM}\right)^{\sigma}=\left(\left(Q\left(A^{\sigma}\right)^{MW}\right)^{\sigma}\right)^{\sigma} = Q\left(A^{\sigma}\right)^{MW}.
$$}
\end{proof}
}

 {We also extend the six defect functions to the full domain. First, we define for each $(x,y)\in [0,1]$ four sets of rectangles
\begin{align*}
    \R_{\searrow}(x,y)&=\left\{[x,u]\times [y,v];\ x\leq u\leq 1, y\leq v\leq 1\right\},\\
    \R_{\swarrow}(x,y)&=\left\{[x,u]\times [v,y];\ x\leq u\leq 1, 0\leq v\leq y\right\},\\
    \R_{\nwarrow}(x,y)&=\left\{[u,x]\times [v,y];\ 0\leq u\leq x, 0\leq v\leq y\right\},\\
    \R_{\nearrow}(x,y)&=\left\{[u,x]\times [y,v];\ 0\leq u\leq x, y\leq v\leq 1\right\}. % {\text{Needs correction.}}
\end{align*}
Then the directional defect functions are defined for each $(x,y)\in [0,1]$ by
\begin{align*}
    D_{\searrow}^Q(x,y)&=\min\{0,V_Q(R);\ R\in\R_{\searrow}(x,y)\},\\
    D_{\swarrow}^Q(x,y)&=\min\{0,V_Q(R);\ R\in\R_{\swarrow}(x,y)\},\\
    D_{\nwarrow}^Q(x,y)&=\min\{0,V_Q(R);\ R\in\R_{\nwarrow}(x,y)\},\\
    D_{\nearrow}^Q(x,y)&=\min\{0,V_Q(R);\ R\in\R_{\nearrow}(x,y)\},
\end{align*}
and the main and the opposite defect function are given by
\begin{align*}
    D_M^Q(x,y)&=\min\left\{D_{\searrow}^Q(x,y),D_{\nwarrow}^Q(x,y)\right\},\\
    D_O^Q(x,y)&=\min\left\{D_{\swarrow}^Q(x,y),D_{\nearrow}^Q(x,y)\right\}.
\end{align*}
}
 {Note again, that all our arrows in the notation above are rotated by $90^{\circ}$ as compared to the arrows in \cite{DS-PMK}, since they correspond to directions in the matrices rather than directions in the coordinate system in the domain of quasi-copulas.}
 { 
\begin{example}\label{ex:sigma}
Consider the ASM
\[
A=\begin{pmatrix}
0 & 1 & 0 \\
1 & -1 & 1 \\
0 & 1 & 0
\end{pmatrix}.
\]
Let $Q(A)$ be the discrete quasi-copula associated with $A$, and let $Q^{MW}$ %and $Q^{BL}$ 
denote the MW 
extension of $Q(A)$.%, respectively.
\begin{figure}[htbp]
    \centering
    \includegraphics[width=0.7\textwidth]{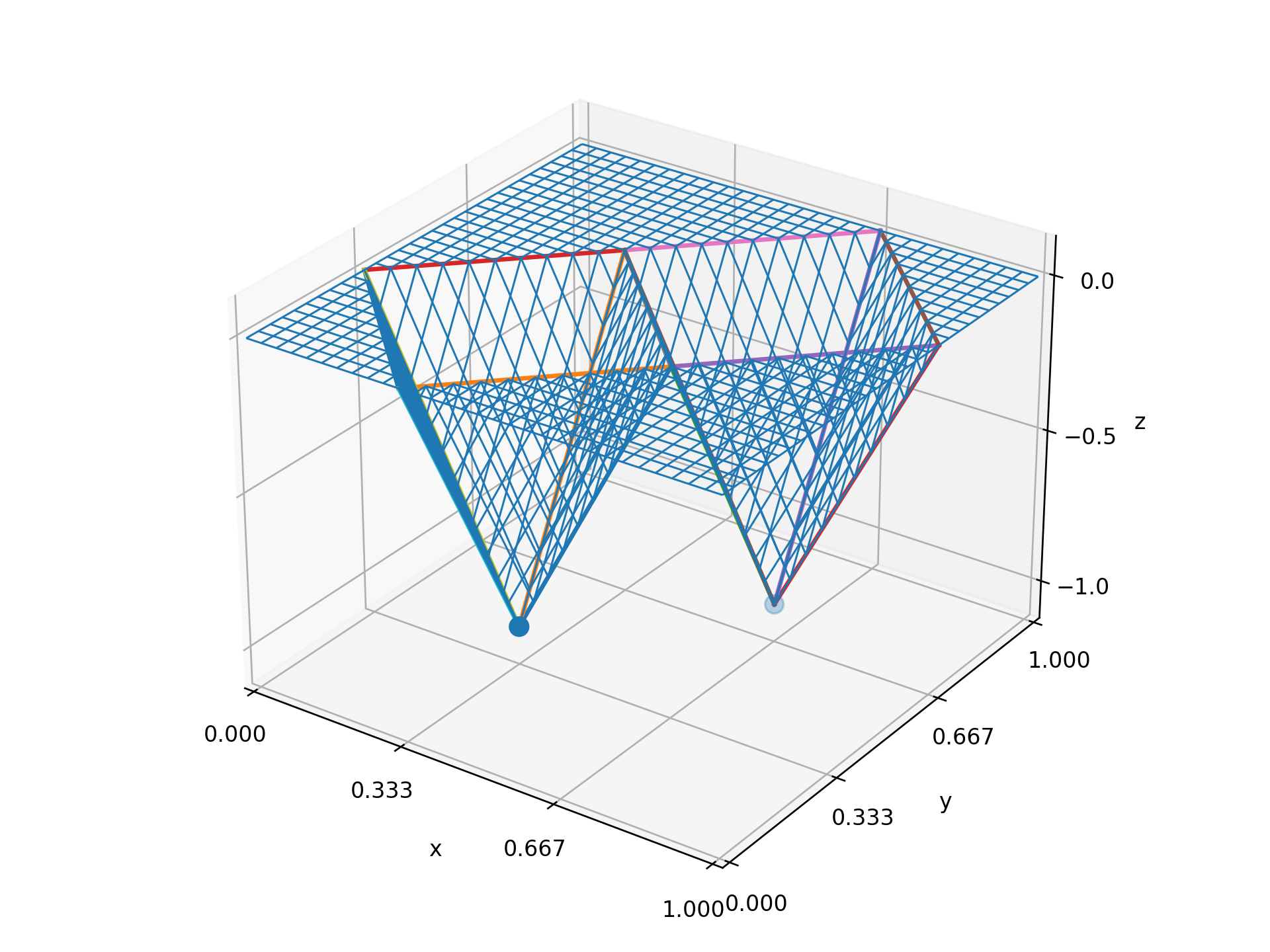}
    \caption{ {Defect of the MW-extension of the discrete quasi-copula considered in Example~\ref{ex:sigma}. The surface follows planar facets toward the minimal values $-1$.}}
    \label{fig:mw_defect}
\end{figure}
The defect of the MW-extension is piecewise linear and is depicted in Figure~\ref{fig:mw_defect}.
\end{example}}

 {The relation between the defect function of a quasi-copula $Q$ and its flipped quasi-copula $Q^{\sigma}$ is described in the following lemma. For a rectangle $R=[x,x']\times [y,y']$, where $x\le x'$ and $y\le y'$, we denote by $R^{\sigma}$ the flipped rectangle $[x,x']\times [1-y,1-y']$.}

\begin{lemma}\label{lem:D_flipped}
     {Suppose that $Q$ is a quasi-copula and $R=[x,x']\times [y,y']$ a rectangle. Then 
    \[
    V_Q(R)=V_{Q^{\sigma}} \left(R^{\sigma} \right). 
    \]
    For the defect functions we have
    \begin{align*}
        D_{\searrow}^{Q}(x,y)&=D_{\swarrow}^{Q^{\sigma}}(x,1-y),\\
        D_{\nwarrow}^{Q}(x,y)&=D_{\nearrow}^{Q^{\sigma}}(x,1-y),\\
        D_M^{Q}(x,y)&=D_O^{Q^{\sigma}}(x,1-y).
    \end{align*}
    }
\end{lemma}

\begin{proof}
     {Mapping $R\to R^{\sigma}$ is a bijection from the set of rectangles $\R_{\searrow}(x,y)$ to the set of rectangles $\R_{\swarrow}(x,1-y)$ and from the set of rectangles $\R_{\nwarrow}(x,y)$ to the set of rectangles $\R_{\nearrow}(x,1-y)$. The volume function satisfies
    \begin{align*}
    V_{Q^{\sigma}} \left(R^{\sigma} \right)&=Q^{\sigma}(x,1-y')+Q^{\sigma}(x',1-y)-Q^{\sigma}(x',1-y')-Q^{\sigma }(x,1-y)\\ 
    &=x-Q(x,y')+x'-Q(x',y)-x'+Q(x',y')-x+Q(x,y)\\
    &=V_Q(R).
    \end{align*}
    Then
    \begin{align*}
        D_{\searrow}^{Q}(x,y)&=\min\{0,V_Q(R);\ R\in\R_{\searrow}(x,y)\}\\
        &=\min\{0,V_{Q^{\sigma}}(R^{\sigma});\ R^{\sigma} \in\R_{\swarrow}(x,1-y)\}\\
        &=D_{\swarrow}^{Q^{\sigma}}(x,1-y)
    \end{align*}
    and
    \begin{align*}
        D_{\nwarrow}^{Q}(x,y)&=\min\{0,V_Q(R);\ R\in\R_{\nwarrow}(x,y)\}\\
        &=\min\{0,V_{Q^{\sigma}}(R^{\sigma});\ R^{\sigma} \in\R_{\nearrow}(x,1-y)\}\\
        &=D_{\nearrow}^{Q^{\sigma}}(x,1-y).
    \end{align*}
    Finally,
    \begin{align*}
        D_M^Q(x,y)&=\min\left\{D_{\searrow}^Q(x,y),D_{\nwarrow}^Q(x,y)\right\}\\
        &=\min\left\{D_{\swarrow}^{Q^{\sigma} } (x,1-y),D_{\nearrow}^{Q^{\sigma} } (x,1-y)\right\}\\
        &=D_O^{Q^{\sigma} }(x,1-y).
    \end{align*}
    }
\end{proof}

 {We now focus on extensions of discrete quasi-copulas of minimal range, i.e., on extensions of those discrete quasi-copulas that correspond to ASMs.}

\begin{proposition}\label{prop:NEdefectOP}
    Let $A \in \A_n$, $\overline{Q}=Q(A)^{MW}$, and $(x,y) \in [0,1]^2$, such that $D_\searrow^{\overline{Q}}(x,y)$ is nonzero. Then there exists a point $(x',y') \in I_n^2$ with $x'>x$ and $y'>y$ such that
    $$D_\searrow^{\overline{Q}}(x,y)=V_{\overline{Q}}([x,x'] \times [y,y']),$$
    i.e., the directional defect of quasi-copula $\overline{Q}$ at point $(x,y)$ is attained on a rectangle with the opposite vertex in $I_n^2$. 
\end{proposition}

\begin{proof}
    Since $D_\searrow^{\overline{Q}}(x,y) \ne 0$ there exists a point $(z_0,w_0) \in [0,1]^2$ with $z_0>x$ and $w_0>y$ such that $D_\searrow^{\overline{Q}}(x,y)=V_{\overline{Q}}([x,z_0] \times [y,w_0])<0$.
    Let $i,j \in \{1,2,\ldots,n\}$ be integers such that $\frac{i-1}{n} \le z_0 \le \frac{i}{n}$ and $\frac{j-1}{n} \le w_0 \le \frac{j}{n}$.

    We investigate how $V_{\overline{Q}}([x,z] \times [y,w])$ is changing when $x,y,z$ are fixed and $w$ runs from $\frac{j-1}{n}$ to $\frac{j}{n}$. Note that
    \begin{equation}\label{eq:defect_vol}
        V_{\overline{Q}}([x,z] \times [y,w])=\overline{Q}(z,w)-\overline{Q}(x,w)-\overline{Q}(z,y)+\overline{Q}(x,y)
    \end{equation}
    and the last two terms are constant.
    We first investigate the behavior of $\overline{Q}(z,w)$ by considering several cases, depending on $V_{\overline{Q}}([\frac{i-1}{n},\frac{i}{n}] \times [\frac{j-1}{n},\frac{j}{n}])$.
    Since $\overline{Q}$ is induced by an alternating sign matrix, we only have three possible cases.
\begin{description}
    \item[Case 1] $V_{\overline{Q}}([\frac{i-1}{n},\frac{i}{n}] \times [\frac{j-1}{n},\frac{j}{n}])=0$.
    In this case $\overline{Q}$ has no mass in $[\frac{i-1}{n},\frac{i}{n}] \times [\frac{j-1}{n},\frac{j}{n}]$, so that $V_{\overline{Q}}([\frac{i-1}{n},z] \times [\frac{j-1}{n},w])=0$ or equivalently
    $$\overline{Q}(z,w)=\overline{Q}(\tfrac{i-1}{n},w)-\overline{Q}(\tfrac{i-1}{n},\tfrac{j-1}{n})+\overline{Q}(z,\tfrac{j-1}{n})=V_{\overline{Q}}(R_w)+\overline{Q}(z,\tfrac{j-1}{n}),$$
    where $R_w=[0,\tfrac{i-1}{n}] \times [\tfrac{j-1}{n},w]$. Note that $\overline{Q}(z,\tfrac{j-1}{n})$ is constant, while $V_{\overline{Q}}(R_w)$ is proportional to $V_{\overline{Q}}([0,\tfrac{i-1}{n}] \times [\tfrac{j-1}{n},\tfrac{j}{n}])$, because all the mass of $\overline{Q}$ is distributed uniformly along the diagonals of the squares. In particular,
    \begin{equation}\label{eq:R_w}
        V_{\overline{Q}}(R_w)=\frac{w-\frac{j-1}{n}}{\frac{1}{n}} \cdot V_{\overline{Q}}([0,\tfrac{i-1}{n}] \times [\tfrac{j-1}{n},\tfrac{j}{n}]).
    \end{equation}
    Since $\overline{Q}$ is induced by an ASM, we have $V_{\overline{Q}}([0,\tfrac{i-1}{n}] \times [\tfrac{j-1}{n},\tfrac{j}{n}]) \in \{0,\frac{1}{n},-\frac{1}{n}\}$. Furthermore, this volume has to be nonnegative since $\overline{Q}$ is a quasi-copula, so that
    \begin{equation}\label{eq:vol}
        V_{\overline{Q}}([0,\tfrac{i-1}{n}] \times [\tfrac{j-1}{n},\tfrac{j}{n}]) \in \{0,\tfrac{1}{n}\}.
    \end{equation}
    Combining all the above we infer that
    $V_{\overline{Q}}(R_w)$, and hence also $\overline{Q}(z,w)$, is linear in $w$ with slope either $0$ or $1$.

    \item[Case 2] $V_{\overline{Q}}([\frac{i-1}{n},\frac{i}{n}] \times [\frac{j-1}{n},\frac{j}{n}])=\frac{1}{n}$.
    In this case $\overline{Q}$ has positive mass $\frac{1}{n}$ distributed uniformly along the main diagonal $z-\frac{i-1}{n}=w-\frac{j-1}{n}$ of the square $[\frac{i-1}{n},\frac{i}{n}] \times [\frac{j-1}{n},\frac{j}{n}]$, so that
    $$V_{\overline{Q}}([\tfrac{i-1}{n},z] \times [\tfrac{j-1}{n},w])=\begin{cases}
        w-\frac{j-1}{n}, & \text{if } w \le z-\frac{i-j}{n},\\
        z-\frac{i-1}{n}, & \text{if } w \ge z-\frac{i-j}{n}.
    \end{cases}$$
    Consequently,
    $$\overline{Q}(z,w)=V_{\overline{Q}}(R_w)+\overline{Q}(z,\tfrac{j-1}{n})+\begin{cases}
        w-\frac{j-1}{n}, & \text{if } w \le z-\frac{i-j}{n},\\
        z-\frac{i-1}{n}, & \text{if } w \ge z-\frac{i-j}{n}.
    \end{cases}$$
    But since $\overline{Q}$ is induced by an ASM and $V_{\overline{Q}}([\frac{i-1}{n},\frac{i}{n}] \times [\frac{j-1}{n},\frac{j}{n}])$ is positive, we must have $V_{\overline{Q}}([0,\tfrac{i-1}{n}] \times [\tfrac{j-1}{n},\tfrac{j}{n}])=0$ by \eqref{eq:vol}, and hence, $V_{\overline{Q}}(R_w)=0$ by \eqref{eq:R_w}.
    We conclude that $\overline{Q}(z,w)$ is piecewise linear in $w$ with slope $1$ at the start and slope $0$ at the end.

    \item[Case 3] $V_{\overline{Q}}([\frac{i-1}{n},\frac{i}{n}] \times [\frac{j-1}{n},\frac{j}{n}])=-\frac{1}{n}$.
    In this case $\overline{Q}$ has negative mass $-\frac{1}{n}$ distributed along the opposite diagonal     $z+w=\frac{i+j-1}{n}$ of the square $[\frac{i-1}{n},\frac{i}{n}] \times [\frac{j-1}{n},\frac{j}{n}]$, so that
    $$V_{\overline{Q}}([\tfrac{i-1}{n},z] \times [\tfrac{j-1}{n},w])=-\begin{cases}
        0, & \text{if } w \le \frac{i+j-1}{n}-z,\\
        w+z-\frac{i+j-1}{n}, & \text{if } w \ge \frac{i+j-1}{n}-z,
    \end{cases}$$
    and
    $$\overline{Q}(z,w)=V_{\overline{Q}}(R_w)+\overline{Q}(z,\tfrac{j-1}{n})-\begin{cases}
        0, & \text{if } w \le \frac{i+j-1}{n}-z,\\
        w+z-\frac{i+j-1}{n}, & \text{if } w \ge \frac{i+j-1}{n}-z.
    \end{cases}$$
    Since $\overline{Q}$ is a quasi-copula and $V_{\overline{Q}}([\frac{i-1}{n},\frac{i}{n}] \times [\frac{j-1}{n},\frac{j}{n}])$ is negative, $V_{\overline{Q}}([0,\tfrac{i-1}{n}] \times [\tfrac{j-1}{n},\tfrac{j}{n}])$ must be positive. Hence, $V_{\overline{Q}}([0,\tfrac{i-1}{n}] \times [\tfrac{j-1}{n},\tfrac{j}{n}])=\frac{1}{n}$ by \eqref{eq:vol} and $V_{\overline{Q}}(R_w)=w-\frac{j-1}{n}$ by \eqref{eq:R_w}.
    Consequently,
    $$\overline{Q}(z,w)=\overline{Q}(z,\tfrac{j-1}{n})+\begin{cases}
        w-\frac{j-1}{n}, & \text{if } w \le \frac{i+j-1}{n}-z,\\
        \frac{i}{n}-z, & \text{if } w \ge \frac{i+j-1}{n}-z.
    \end{cases}$$
    We conclude that $\overline{Q}(z,w)$ is piecewise linear in $w$ with slope $1$ at the start and slope $0$ at the end.
\end{description}
In all cases there exists a point $w_1 \in [\frac{j-1}{n},\frac{j}{n}]$ such that $\overline{Q}(z,w)$ is piecewise linear in $w$ with slope $1$ on $[\frac{j-1}{n},w_1]$ and slope $0$ on $[w_1,\frac{j}{n}]$ (in Case 1 the point $w_1$ is one of the endpoints of the interval $[\frac{j-1}{n},\frac{j}{n}]$).
Replacing $z$ by $x$, the same conclusion can be made for $\overline{Q}(x,w)$, i.e., there exists a point $w_2 \in [\frac{j-1}{n},\frac{j}{n}]$ such that $\overline{Q}(x,w)$ is piecewise linear in $w$ with slope $1$ on $[\frac{j-1}{n},w_2]$ and slope $0$ on $[w_2,\frac{j}{n}]$.

From the above we infer that $\overline{Q}(z,w)-\overline{Q}(x,w)$ is piecewise linear in $w$ with slope $0$ at the start, slope $0$ at the end, and slope either $1$ or $-1$ in the middle.
In particular, it is monotone. By \eqref{eq:defect_vol}, we conclude that $V_{\overline{Q}}([x,z] \times [y,w])$ is a monotone function of $w \in [\frac{j-1}{n},\frac{j}{n}]$, as well as a monotone function of $z \in [\frac{i-1}{n},\frac{i}{n}]$ by a similar argument.

For a fixed $z$, the volume $V_{\overline{Q}}([x,z] \times [y,w])$ attains its minimum when $w$ is one of the endpoints of the interval $[\frac{j-1}{n},\frac{j}{n}]$. Denote this endpoint by $y'$, and note that $y'>y$.
We remark that if $y\in [\frac{j-1}{n},\frac{j}{n}]$, we have to restrict our attention to $w \in [y,\frac{j}{n}]$, but the minimum cannot be attained at $y$,
because in that case the minimum would be $0$, which would imply $D_\searrow^{\overline{Q}}(x,y)=0$, contradicting our assumption. So, in this situation, the minimum is attained at the endpoint $y'=\frac{j}{n}$.
Furthermore, taking $w=y'$, we infer similarly that $V_{\overline{Q}}([x,z] \times [y,y'])$ attains its minimum when $z$ is one of the endpoints the interval $[\frac{i-1}{n},\frac{i}{n}]$. Denote this endpoint by $x'$ and note that $x'>x$.
We conclude that
$$\begin{aligned}
    D_\searrow^{\overline{Q}}(x,y) &=V_{\overline{Q}}([x,z_0] \times [y,w_0])\\
    &=\min \{V_{\overline{Q}}([x,z] \times [y,w]) \mid \max\{x,\tfrac{i-1}{n}\}\le z \le \tfrac{i}{n}, \max\{y,\tfrac{j-1}{n}\}\le w \le \tfrac{j}{n}\}\\
    &=\min \{V_{\overline{Q}}([x,z] \times [y,y']) \mid \max\{x,\tfrac{i-1}{n}\}\le z \le \tfrac{i}{n}\}\\
    &=V_{\overline{Q}}([x,x'] \times [y,y'])
\end{aligned}$$
and $(x',y') \in I_n^2$.
\end{proof}

 {\begin{proposition}\label{prop:D_M_one}
    Let $A \in \A_n$ and $\overline{Q}=Q(A)^{MW}$. Furthermore, let $(x,y) \in [0,1]^2$ be a point such that $D_M^{\overline{Q}}(x,y) \ne 0$. Then there exists a point $(x',y') \in I_n^2$ with $(x'-x)(y'-y)>0$ such that
    $$D_M^{\overline{Q}}(x,y)=V_{\overline{Q}}(R_{(x,y),(x',y')}),$$
    i.e., the main defect of quasi-copula $\overline{Q}$ at point $(x,y)$ is attained on a rectangle with the opposite vertex in $I_n^2$.
\end{proposition}}

 {\begin{proof}
    By definition $$D_M^{\overline{Q}}(x,y)=\min\left\{D_\searrow^{\overline{Q}}(x,y),D_\nwarrow^{\overline{Q}}(x,y)\right\}.$$
    If $D_\searrow^{\overline{Q}}(x,y) \le D_\nwarrow^{\overline{Q}}(x,y)$, then the conclusion follows from Proposition~\ref{prop:NEdefectOP}. So suppose $D_\searrow^{\overline{Q}}(x,y) > D_\nwarrow^{\overline{Q}}(x,y)$, so that $D_M^{\overline{Q}}(x,y)=D_\nwarrow^{\overline{Q}}(x,y)$.
    Denote by $\widehat{Q}$ the survival quasi-copula of $\overline{Q}$, induced by the reflection $(x,y) \mapsto (1-x,1-y)$.
    Since the reflection maps a rectangle $R \in \mathcal{R}_\nwarrow(x,y)$ into a rectangle $\widehat{R} \in \mathcal{R}_\searrow(1-x,1-y)$, we have
    $D_\nwarrow^{\overline{Q}}(x,y)=D_\searrow^{\widehat{Q}}(1-x,1-y)$.
    Furthermore, since the reflection maps main diagonals of squares to main diagonals of reflected squares and opposite diagonals of squares to opposite diagonals of reflected squares, we have $\widehat{Q}=Q(B)^{MW}$, for some $B \in \A_n$. In fact, $B=PAP$, where $P$ is a permutation matrix induced by permutation $\pi(i)=n+1-i$.
    By Proposition~\ref{prop:NEdefectOP} there exists $(x'',y'') \in I_n^2$ with $x''>1-x$ and $y''>1-y$ such that
    $$D_\searrow^{\widehat{Q}}(1-x,1-y)=V_{\widehat{Q}}([1-x,x''] \times [1-y,y''])=V_{\overline{Q}}([1-x'',x] \times [1-y'',y]).$$
    If we let $x'=1-x''$ and $y'=1-y''$, then $D_M^{\overline{Q}}(x,y)=V_{\overline{Q}}([x',x] \times [y',y])$, with $x'<x$, $y'<y$, and $(x',y') \in I_n^2$. This concludes the proof.
\end{proof}
}

 {Lemmas \ref{lem:flip} and \ref{lem:D_flipped} together with Propositions \ref{prop:NEdefectOP} and \ref{prop:D_M_one} imply the following result.}

\begin{corollary}\label{cor:OPdefect}
       {Let $A \in \A_n$, $\overline{Q}=Q(A)^{WM}$, and $(x,y) \in [0,1]^2$, such that $D_\swarrow^{\overline{Q}}(x,y)$ is nonzero. Then there exists a point $(x',y') \in I_n^2$ with $x'>x$ and $y'<y$ such that
    $$D_\swarrow^{\overline{Q}}(x,y)=V_{\overline{Q}}([x,x'] \times [y',y]),$$
    i.e., the directional defect of quasi-copula $\overline{Q}$ at point $(x,y)$ is attained on a rectangle with the opposite vertex in $I_n^2$. }

      {Furthermore, let $(x,y) \in [0,1]^2$ be a point such that $D_O^{\overline{Q}}(x,y) \ne 0$. Then there exists a point $(x',y') \in I_n^2$ with $(x'-x)(y'-y)<0$ such that
    $$D_O^{\overline{Q}}(x,y)=V_{\overline{Q}}(R_{(x,y),(x',y')}),$$
    i.e., the opposite defect of quasi-copula $\overline{Q}$ at point $(x,y)$ is attained on a rectangle with the opposite vertex in $I_n^2$.}
\end{corollary}

\begin{proof}
     { By Lemma \ref{lem:flip} it follows that $\left(Q(A)^{WM}\right)^{\sigma}=Q(A^{\sigma})^{MW}$. By Lemma \ref{lem:D_flipped}, the flipping operation implies that $D^{Q^{\sigma}}_{\swarrow}(x,y)=D^Q_{\searrow}(x,1-y)$ for any quasi-copula $Q$. By Proposition \ref{prop:NEdefectOP} it follows that the nonzero defect $D_\searrow^{\overline{Q}}(x,1-y)$ is attained on a rectangle with the opposite vertex in $I_n^2$. Hence, so is $D^{Q^{\sigma}}_{\swarrow}(x,y)$. }
    
     {In a similar way, Proposition \ref{prop:D_M_one} implies that the nonzero opposite defect $D_O^{\overline{Q}}(x,y)$ is attained on a rectangle with the opposite vertex in $I_n^2$.
    }
\end{proof}

\section{Extensions in the case of dense ASMs}\label{sect:ext_dense}

 {In this section we focus our attention again on dense ASMs. We study the MW and WM-extensions of the corresponding discrete quasi-copulas and imprecise copulas. As the main result we show that certain pairs of extensions of quasi-copulas that correspond to dense ASMs form a full range self-dual imprecise copulas. }
To this purpose, we first recall the definition of a self-dual imprecise copula from \cite{KoPe-DIC}.

 {
\begin{definition}
\label{def:imco-f}
    An imprecise copula $(P,Q)$ is \emph{self-dual} if $P_M=Q$ and $Q_O=P$. 
\end{definition}}

It follows from Definition~\ref{def:imco-f} that an imprecise copula $(P,Q)$ is self-dual if and only if $D_M^P=D_O^Q=P-Q$. 

 {Omladič and Stopar \cite[Thm. 7]{FinalS} show that for an imprecise copula $(P,Q)$ the pairs $(P,P_M)$ and $(Q_O,Q)$ are again imprecise copulas. In particular, $P \le (P_M)_O \le P_M \le Q$.
Consequently, $(P,P_M)$ is an imprecise copula inside the original imprecise copula $(P,Q)$, and  $((P_M)_O,P_M)$ is an imprecise copula inside $(P,P_M)$.
Iterating the process further, we get a sequence of embedded imprecise copulas $(P_k,Q_k)$ given by $P_0=P$, $Q_0=Q$ and $P_k=(Q_{k})_O, Q_k=(P_{k-1})_M$. The limiting pair $(\overline{P},\overline{Q})$ is also an imprecise copula which satisfies the equalities $(\overline{P})_M=\overline{Q}$ and $(\overline{Q})_O=\overline{P}$, so it is self-dual. See \cite[Prop. 10]{FinalS} for the results on the iteration process.}

Our first theorem describes the main defect of an $MW$-extension of a dense ASM quasi-copula.

\begin{theorem}\label{thm:denseMdefect}
 {Let $A \in \A_n$ be a dense ASM and $D_M^{Q(A)}$ the main defect of the associated discrete quasi-copula $Q(A)$. Then, the main defect of $\overline{Q}=Q(A)^{MW}$ is a piecewise linear extension of $D_M^{Q(A)}$, with an inverted  {four-sided} pyramidal shape around each point where $D_M^{Q(A)}=-1$, where it attains its minimum value $-\frac{1}{n}$.}
\begin{proof}
{ Let $R=[\tfrac{i-1}{n},\tfrac{i}{n}] \times [\tfrac{j-1}{n},\tfrac{j}{n}]$ be a square with vertices in $I_n^2$. Since $A$ is a dense ASM, $V_{\overline{Q}}(R)\geq -\tfrac{1}{n}$ (See Prop. 4.5 of \cite{KoPe-DIC}). From Prop~\ref{prop:D_M_one}, for any rectangle $R_c$ in $[0,1]^2$, the defect is attained on some points on $I_n^2$. Therefore, we can conclude that $V_{\overline{Q}}(R_c) \geq -\tfrac{1}{n}$. }

{We now show how $D_M^{\overline{Q}}$ can be derived as a  {piecewise} linear extension of $D_M^{Q(A)}$.  {An instance of the resulting shape can be seen in Figure~\ref{fig:mw_defect}.} We consider three possible cases. For all cases, we assume $ (x, y) \in R$ and denote as $R_k$, with $k=1,\ldots, 6$ the six adjacent squares to $R$. Namely, we have $R_1=[\tfrac{i}{n},\tfrac{i+1}{n}] \times [\tfrac{j-1}{n},\tfrac{j}{n}]$, $R_2=[\tfrac{i-1}{n},\tfrac{i}{n}] \times [\tfrac{j}{n},\tfrac{j+1}{n}]$, $R_3=[\tfrac{i}{n},\tfrac{i+1}{n}] \times [\tfrac{j}{n},\tfrac{j+1}{n}]$, $R_4=[\tfrac{i-2}{n},\tfrac{i-1}{n}] \times [\tfrac{j-1}{n},\tfrac{j}{n}]$, $R_5=[\tfrac{i-1}{n},\tfrac{i}{n}] \times [\tfrac{j-2}{n},\tfrac{j-1}{n}]$, $R_6=[\tfrac{i-2}{n},\tfrac{i-1}{n}] \times [\tfrac{j-2}{n},\tfrac{j-1}{n}]$. }

\begin{description}
\item[Case 1] We consider the case when $V_{\overline{Q}}(R) = 0$. If $V_{\overline{Q}}(R_3) = V_{\overline{Q}}(R_6)=0$, then $D_M^{Q(A)}(x,y) = 0$.
If $V_{\overline{Q}}(R_3)=-\tfrac{1}{n}$, then $V_{\overline{Q}}(R_1)=V_{\overline{Q}}(R_2)=\tfrac{1}{n}$ and $V_{\overline{Q}}(R_6)=0$ because A is dense. 
In such a case, the minimum defect is attained at the vertex $(\tfrac{i+1}{n},\tfrac{j+1}{n})$, i.e., it corresponds to the volume of the rectangle $R_a=[x,\tfrac{i+1}{n}] \times [y,\tfrac{j}{n}]$. Such a volume is equal to
\[D_M^{\overline{Q}}(x,y)=V_{\overline{Q}}(R_a)=V_{\overline{Q}}(\ival{x,\tfrac{i+1}{n}}\times \ival{y,\tfrac{j}{n}}) + V_{\overline{Q}}(\ival{x,\tfrac{i}{n}}\times \ival{\tfrac{j}{n},\tfrac{j+1}{n}}) + V_{\overline{Q}}(R_3) = \]
\[ =(\tfrac{1}{n}-x)+(\tfrac{1}{n}-y)-\tfrac{1}{n}= \tfrac{1}{n}-x-y,\]
Therefore $D_M^{\overline{Q}}(x,y) \leq 0$ if and only if $x + y \geq \tfrac{1}{n}$.

Similarly, if $V_{\overline{Q}}(R_6)=-\tfrac{1}{n}$, then $V_{\overline{Q}}(R_4)=V_{\overline{Q}}(R_5)=\tfrac{1}{n}$ and $V_{\overline{Q}}(R_3)=0$ because A is dense. In such a case, the minimum defect is attained at the vertex $(\tfrac{i-2}{n},\tfrac{j-2}{n})$, i.e., it corresponds to the volume of the rectangle $R_b=[\tfrac{i-2}{n},x] \times [\tfrac{j-2}{n},y]$.
The volume of $R_b$ can be computed as follows:
\[D_M^{\overline{Q}}(x,y)=V_{\overline{Q}}(R_b)=V_{\overline{Q}}(\ival{\tfrac{i-1}{n},x}\times \ival{\tfrac{j-2}{n},y}) + V_{\overline{Q}}(\ival{\tfrac{i-2}{n},\tfrac{i-1}{n}}\times \ival{\tfrac{j-1}{n},y}) + V_{\overline{Q}}(R_6) = \]
\[ =(\tfrac{1}{n}-x)+(\tfrac{1}{n}-y)-\tfrac{1}{n}= \tfrac{1}{n}-x-y.\]
Thus, $D_M^{\overline{Q}}(x,y) \leq 0$ if and only if $x + y \geq \tfrac{1}{n}$.
\item[Case 2] We now assume that $V_{\overline{Q}}(R) = -\tfrac{1}{n}$. Since A is dense, $V_{\overline{Q}}(R_1)=V(R_2)=\tfrac{1}{n}$. We distinguish two cases. If $x + y \geq \tfrac{1}{n}$, then the defect is again attained at the vertex $(\tfrac{i+1}{n},\tfrac{j+1}{n})$ as for the previous case. Therefore, it corresponds to the volume of the rectangle $R_a=[x,\tfrac{i+1}{n}] \times [y,\tfrac{j}{n}]$, which results in 
\[D_M^{\overline{Q}}(x,y)= \tfrac{1}{n}-x-y.\]
If $x + y < \tfrac{1}{n}$, then the defect is attained at the vertex $(\tfrac{i}{n},\tfrac{j}{n})$, and it corresponds to the mass of the rectangle 
$R_c=[x,\tfrac{i}{n}] \times [y,\tfrac{j}{n}]$, i.e., \[D_M^{\overline{Q}}(x,y)=-\tfrac{1}{n}+x+y.\]
The defect is zero when $x+y=1$.
\item[Case 3] We now consider the case when $V_{\overline{Q}}(R)=\tfrac{1}{n}$. We first consider the situation when $V_{\overline{Q}}(R_1)=V_{\overline{Q}}(R_2)=-\tfrac{1}{n}$. Here, we can distinguish two cases. First $x \leq y$. In such a case, the defect is attained at the vertex $(\tfrac{i}{n},\tfrac{j+1}{n})$, i.e., it is the volume of the rectangle $R_d=[x,\tfrac{i}{n}] \times [y,\tfrac{j+1}{n}]$. Such a volume can be computed as follows:
\[D_M^{\overline{Q}}(x,y)=V_{\overline{Q}}(R_d)=V_{\overline{Q}}(\ival{x,\tfrac{i}{n}}\times \ival{y,\tfrac{j}{n}}) + V_{\overline{Q}}(\ival{x,\tfrac{i}{n}}\times \ival{\tfrac{j}{n},\tfrac{j+1}{n}}) = \]
\[ =(\tfrac{1}{n}-y) + (-\tfrac{1}{n}+x)= x-y.\]
Now, we consider the case $x \geq y$. In such a case, the defect is attained at the vertex $(\tfrac{i+2}{n},\tfrac{j+1}{n})$. The defect can be computed similarly as in the previous case, and it is equal to
\[D_M^{\overline{Q}}(x,y)=(\tfrac{1}{n}-x) + (-\tfrac{1}{n}+y)= y-x.\]
Again, the defect is equal to zero exactly on the main diagonal $x=y$.

Now, we need to distinguish a few other cases. In particular, we look at situations where $V_{\overline{Q}}(R_1)=V_{\overline{Q}}(R_4)=0$. In such a case, $V_{\overline{Q}}(R_2)=-\tfrac{1}{n}$ and $V_{\overline{Q}}(R_3)=\tfrac{1}{n}$. Then, if $x\leq y$, we are back to the first case considered, and 
\[D_M^{\overline{Q}}(x,y) = (\tfrac{1}{n}-y) + (-\tfrac{1}{n}+x)= x-y.\] Otherwise, the defect is zero.
The same reasoning holds if we have $V_{\overline{Q}}(R_2)=0$, i.e., we are at the boundary of the dense ASM from below.
\end{description}
\end{proof}
\end{theorem}

 {Since $A^{\sigma}$ is also a dense ASM if $A$ is a dense ASM we obtain a dual result to Theorem \ref{thm:denseMdefect}. Its proof follows the same arguments, we only need to reverse the directions from the main to the opposite. Alternatively, we may apply Lemma \ref{lem:flip} and Corollary \ref{cor:OPdefect}. }

\begin{theorem}\label{thm:denseOdefect}
     {Let $A \in \A_n$ be a dense ASM and $D_O^{Q(A)}$ the opposite defect of the associated discrete quasi-copula $Q(A)$. Then, the opposite defect of $\overline{Q}=Q(A)^{WM}$ is a piecewise linear extension of $D_O^{Q(A)}$, with an inverted  {four-sided} pyramidal shape around each point where $D_O^{Q(A)}=-1$, where it attains its minimum value $-\frac{1}{n}$.}
\end{theorem}

 {Next we state the main result of this section. Note that the notation $F_n^k$ is introduced at the beginning of Section \ref{five}.}

\begin{theorem}\label{thm:dual_pair}
     {Suppose that $F_n^k$ is a dense ASM, where $n\ge 4$ and $3\le k\le n-1$, and that $\overline{Q}=Q(F_n^k)^{MW}$ is its MW-extension. Then $\overline{Q}_M=Q\left(F_n^{k-1}\right)^{WM}$ and 
    $$(\overline{Q},\overline{Q}_M)=\left(Q\left(F_n^k\right)^{MW},Q\left(F_n^{k-1}\right)^{WM}\right)$$ 
    is a full domain self-dual imprecise copula.}
\end{theorem}

\begin{proof}
     {Recall that by \cite[Thm. 4.8 and Cor. 4.9]{KoPe-DIC} we have $D_M^{F_n^{k}} =D_O^{F^{k-1}_n }$ for any $n\ge 4$ and $3\le k\le n-1$. Theorems \ref{thm:denseMdefect} and \ref{thm:denseOdefect} imply that the main defect of $\overline{Q}$ and the opposite defect of the $\left(Q\left(F_n^{k-1}\right)\right)^{WM}$ coincide. Hence, $(\overline{Q},\overline{Q}_M)$ is a full domain self-dual imprecise copula.}
\end{proof}

\section{ASM imprecise copulas need not avoid sure loss}\label{DIC}

In this section, we move away from the case of dense ASMs and analyze the general case of ASM imprecise copulas, namely discrete imprecise copulas whose discrete quasi-copulas correspond to ASMs.  In particular, we show that ASM imprecise copulas do not in general avoid sure loss.

We extend the definition of a self-dual imprecise quasi-copula to the discrete setting.

\begin{definition}
\label{def:imco}
    A discrete imprecise copula $(P,Q)$ is \emph{self-dual} if $P_M=Q$ and $Q_O=P$. Alternatively, we say that a pair of discrete quasi-copulas $(P,Q)$ is \emph{a dual pair} if $(P,Q)$ is a self-dual discrete imprecise copula. Furthermore, if $P,Q\in\Q_n$ and $A=A(P)$, $B=A(Q)$ then we say that a pair of ASMs $(A,B)$ is \emph{a dual pair} if $(P,Q)$ is a self-dual discrete imprecise copula, i.e., if $(P,Q)$ is a dual pair of discrete quasi-copulas.  
\end{definition}

An example of a discrete imprecise copula $(P,Q)$, where both $P$ and $Q$ are quasi-copulas of minimal range of size $7$, is given by Omladi\v{c} and Stopar \cite[Ex. 11]{FinalS}. The same example gives a self-dual discrete imprecise copula which is then extended to a dual pair of proper full-domain quasi-copulas in \cite[Sect. 5]{K-BKOS}. In \cite{KoPe-DIC} a further generalization of this example that yields self-dual discrete imprecise copulas to finer grids of size $n\ge 8$ is given.

In \cite{FinalS}, it is shown that there exists a discrete imprecise copula $(P,Q)$ defined on $I_{n}^2$ (with $n=10$) that does not avoid sure loss, that is, there is no discrete copula $C$ on $I_n^2$ such that $P(x,y) \le C(x,y) \le Q(x,y)$ for all $(x,y) \in I_n^2$.
By modifying the construction from \cite{FinalS}, we obtain the following stronger result.

\begin{proposition}
    For any $n \ge 17$ there exists a self-dual discrete imprecise copula $(P,Q)$ for which both $P$ and $Q$ are minimal range discrete quasi-copulas on $L_n^2$ such that $(P,Q)$ does not avoid sure loss.
\end{proposition}
\begin{proof}
     {For a detailed explanation of the notation and the concepts used in this proof, in particular, the definition of the multiplicity $m_R(\mathbf{y})$ of a point $\mathbf{y}$ with respect to a disjoint union $R$ of rectangles, and the definition of the function $L^{(P,Q)}$ determined by two quasi-copulas $P$ and $Q$, we refer the reader to \cite{FinalS}.
    The proof itself is a slight modification of the proof for the example given in \cite[\S 5]{FinalS}.}
    
    We notice that to prove the theorem it suffices to find an example for $n=17$. Then, the general case of $n\geq 17$ can be derived for example by adding rows and columns containing a single entry equal to $1$ and all remaining entries equal to $0$, thereby preserving the alternating sign matrix property. Let $A \in \A_{17}$ be the following alternating sign matrix:
    \newcommand{\cc}{\cellcolor{black!15!white}}
    $$A=\left[
    \begin{array}{ccccccccccccccccc}
    0 & 0 & 0 & 0 & 0 & \bf{1} & 0 & 0 & 0 & 0 & 0 & 0 & 0 & 0 & 0 & 0 & 0 \\
    0 & 0 & 0 & 0 & 0 & 0 & 0 & 0 & \bf{1} & 0 & 0 & 0 & 0 & 0 & 0 & 0 & 0 \\
    0 & 0 & 0 & 0 & 0 & 0 & 0 & 0 & 0 & 0 & \bf{1} & 0 & 0 & 0 & 0 & 0 & 0 \\
    0 & 0 & 0 & 0 & 0 & 0 & 0 & 0 & 0 & 0 & 0 & \bf{1} & 0 & 0 & 0 & 0 & 0 \\
    0 & 0 & 0 & 0 & \bf{1} & \cc 0 & \cc 0 & \cc 0 & \cc -\bf{1} & \cc 0 & \cc 0 & \cc 0 & 0 & 0 & 0 & 0 & \bf{1} \\
    0 & 0 & 0 & 0 & 0 & \cc 0 & 0 & 0 & 0 & \bf{1} & \cc -\bf{1} & \cc 0 & 0 & 0 & 0 & \bf{1} & 0 \\
    0 & 0 & 0 & 0 & 0 & \cc 0 & 0 & 0 & \bf{1} & 0 & \cc 0 & \cc 0 & 0 & 0 & 0 & 0 & 0 \\
    0 & 0 & 0 & \bf{1} & 0 & \cc -\bf{1} & \bf{1} & 0 & 0 & 0 & \cc 0 & \cc 0 & 0 & 0 & 0 & 0 & 0 \\
    0 & 0 & \bf{1} & 0 & 0 & \cc 0 & \cc 0 & \cc 0 & \cc -\bf{1} & \cc 0 & \cc 0 & \cc 0 & 0 & 0 & \bf{1} & 0 & 0 \\
    0 & 0 & 0 & 0 & 0 & \cc 0 & \cc 0 & 0 & 0 & 0 & \bf{1} & \cc -\bf{1} & 0 & \bf{1} & 0 & 0 & 0 \\
    0 & 0 & 0 & 0 & 0 & \cc 0 & \cc 0 & 0 & \bf{1} & 0 & 0 & \cc 0 & 0 & 0 & 0 & 0 & 0 \\
    0 & \bf{1} & 0 & 0 & 0 & \cc 0 & \cc -\bf{1} & \bf{1} & 0 & 0 & 0 & \cc 0 & 0 & 0 & 0 & 0 & 0 \\
    \bf{1} & 0 & 0 & 0 & 0 & \cc 0 & \cc 0 & \cc 0 & \cc -\bf{1} & \cc 0 & \cc 0 & \cc 0 & \bf{1} & 0 & 0 & 0 & 0 \\
    0 & 0 & 0 & 0 & 0 & \bf{1} & 0 & 0 & 0 & 0 & 0 & 0 & 0 & 0 & 0 & 0 & 0 \\
    0 & 0 & 0 & 0 & 0 & 0 & \bf{1} & 0 & 0 & 0 & 0 & 0 & 0 & 0 & 0 & 0 & 0 \\
    0 & 0 & 0 & 0 & 0 & 0 & 0 & 0 & \bf{1} & 0 & 0 & 0 & 0 & 0 & 0 & 0 & 0 \\
    0 & 0 & 0 & 0 & 0 & 0 & 0 & 0 & 0 & 0 & 0 & \bf{1} & 0 & 0 & 0 & 0 & 0 
    \end{array}
    \right].$$
    Define $P_0=Q(A)$ and $Q_0=Q(A)_M$.
    By \cite[Theorem~7]{FinalS}, the pair $(P_0,Q_0)$ is a discrete imprecise copula (note that the proof also applies to discrete quasi-copulas).
    We prove that $(P_0,Q_0)$ does not avoid sure loss by applying the characterization given in \cite[Proposition~16]{FinalS}, using the function $L^{(P_0,Q_0)}$ defined there.
    Let $R$ denote the shaded region in $A$, i.e., the disjoint union of $39$ squares determined by the mesh $L_{17}^2$.
    We claim that $L^{(P_0,Q_0)}(R)<0$.
    The region $R$ has $6$ vertices with positive multiplicity, namely, $(4,5), (5,10), (8,6), (9,11), (12,7), (13,12)$, each with multiplicity $1$, and the defect $D_{M}^{P_0}$ is equal to $-{1}$ at all $6$ vertices.
    Hence,
    \begin{align*}
        L^{(P_0,Q_0)}(R)
        &=\sum_{m_R(\mathbf{y})<0} P_0(\mathbf{y})m_R(\mathbf{y})+ \sum_{m_R(\mathbf{y})>0} Q_0(\mathbf{y}) m_R(\mathbf{y})=\\
        &=\sum_{m_R(\mathbf{y}) \neq 0} P_0(\mathbf{y})m_R(\mathbf{y})+ \sum_{m_R(\mathbf{y})>0}(Q_0(\mathbf{y})- P_0(\mathbf{y})) m_R(\mathbf{y})=\\
        &=V_{P_0}(R)+\sum_{m_R(\mathbf{y})>0}(-D_M^{P_0}(\mathbf{y}))m_R(\mathbf{y})=-{7}+{6}=-{1}.
    \end{align*}
    Therefore, $(P_0,Q_0)$ does not avoid sure loss by \cite[Proposition~16]{FinalS}.
    Now consider the sequences
    $P_i$ and $Q_i$ defined recursively by $P_{i+1}=(Q_i)_O$ and $Q_{i+1}=(P_{i+1})_M$ for all $i \ge 0$.
    By Theorem~\ref{thm:ASM_transformations}, the corresponding matrices $A(P_i)$ and $A(Q_i)$ are ASMs for all $i \geq 0$. 
    It follows from \cite[Proposition~10]{FinalS} (whose proof also applies to discrete quasi-copulas) that the sequences $P_i$ and $Q_i$ converge to some discrete quasi-copulas $P$ and $Q$, respectively, such that $(P,Q)$ is a self-dual discrete imprecise quasi-copula and $P_0 \le P \le Q \le Q_0$.
    Furthermore, both sequences are eventually constant, since the set of all ASMs of size $17 \times 17$ is finite.
    Consequently, $A(P)$ and $A(Q)$ are ASMs as well, and therefore $P$ and $Q$ are minimal range discrete quasi-copulas.
    Since $P_0 \le P \le Q \le Q_0$ and $(P_0,Q_0)$ does not avoid sure loss, neither does $(P,Q)$.
\end{proof}

 {The grid size $n=17$ is the smallest for which we have found an ASM with the corresponding $L$-function attaining a negative value. For smaller grid sizes, the number of proper ASMs is more restricted: for example, for $n=3$ there is only one proper ASM. We do not know whether, for some $n<17$, there exists an ASM for which the corresponding $L$-function attains a negative value. Determining the smallest grid size for which this occurs remains an open problem.}

\section{Conclusion}
In this paper, we study discrete quasi-copulas and discrete imprecise copulas of minimal range, which correspond to ASMs. We show that discrete quasi-copulas of minimal range are invariant under defect transformations, establishing a strong structural stability property of this class.
 {We then focus on dense ASMs and prove that the corresponding discrete imprecise copulas are always coherent and therefore avoid sure loss. We also study extensions of discrete quasi-copulas to the full domain and show that, in the case of dense ASMs, these extensions yield a new family of self-dual imprecise copulas on the full domain.
In contrast to the favorable properties of the dense case, we demonstrate by means of a constructive example that discrete imprecise copulas of minimal range do not, in general, avoid sure loss. Our results thus reveal an important distinction within the class of minimal range imprecise copulas: while invariance under defect transformations holds throughout the entire class, the additional density assumption guarantees coherence and also plays a central role in obtaining self-dual imprecise copulas on the full domain.
The results presented in this paper leave several questions open for future investigation. In particular, it would be interesting to understand how defect transformations interact with order-theoretic properties of alternating sign matrices. For example, one may ask how the ranks of transformed matrices compare with the rank of the original matrix. Another natural question concerns the existence of a copula corresponding to a permutation matrix within an imprecise copula associated with alternating sign matrices that avoids sure loss.
Our results also suggest several questions concerning the relation between the discrete and full-domain settings. In particular, a natural direction for further research concerns the approximation of full-domain quasi-copulas and imprecise copulas by discrete quasi-copulas and discrete imprecise copulas, respectively.}

\section*{Acknowledgments}

The authors would like to thank Damjana Kokol Bukovšek, whose helpful comments improved the readability of the paper. Additionally, we would like to thank two anonymous reviewers and the associate editor for their comments on a previous version of the paper. 

Tomaž Košir acknowledges financial support from ARIS (Slovenian Research and Innovation Agency), research core funding P1-0448 before December 31, 2025, and research core funding P1-0222 and project J1-50002 after January 1, 2026. Nik Stopar acknowledges financial support from ARIS, research core funding P1-0222, and projects J1-70034 and J1-50002.

\section*{Declaration of generative AI and AI-assisted technologies}

No AI tools were used during the preparation of this manuscript.

\bibliographystyle{plain}
\bibliography{references} 

\end{document}